\documentclass[a4paper, 12pt]{amsart}

\usepackage[french, english]{babel}
\usepackage[normalem]{ ulem }
\usepackage{soul}

\usepackage[T1]{fontenc}
\usepackage[utf8]{inputenc}
\usepackage{amsmath,amsfonts,amssymb,amsopn,amscd,amsthm}
\usepackage{gensymb}
\usepackage{comment}
\usepackage{dsfont}
\usepackage{bm}
\usepackage{graphicx}
\usepackage{color}
\usepackage[colorlinks]{hyperref}
\usepackage{epigraph}
\usepackage{enumerate}
\usepackage{xargs}
\usepackage[prependcaption]{todonotes}

\usepackage{marginnote}
\usepackage{todonotes}
\usepackage[left=3.5cm,top=2.5cm,bottom=2cm,right=4cm, marginparwidth=3cm]{geometry}

\title[ Spiking and collapsing in large noise limits of SDEs]
      { Spiking and collapsing in large noise limits of SDEs}
      
\author[Bernardin]{\textsc{C\'edric Bernardin}} 
\address{Universit\'e C\^ote d'Azur, CNRS, LJAD\\
Parc Valrose\\
06108 NICE Cedex 02, France}
\email{{\tt cbernard@unice.fr}}

\author[Chetrite]{\textsc{Rapha\"el Chetrite}}
\address{Universit\'e C\^ote d'Azur, CNRS, LJAD\\
Parc Valrose\\
06108 NICE Cedex 02, France}
\email{{\tt raphael.chetrite@unice.fr}}
      
\author[Chhaibi]{\textsc{Reda Chhaibi}}
\address{Universit\'e Paul Sabatier, Toulouse 3 -- Institut de math\'ematiques de Toulouse (IMT) -- 118, route de Narbonne, 31400, Toulouse, France}
\email{{\tt reda.chhaibi@math.univ-toulouse.fr}}

\author[Najnudel]{\textsc{Joseph Najnudel}} 
\address{University of Bristol -- University Walk, Clifton, Bristol, United Kingdom}
\email{{\tt joseph.najnudel@bristol.ac.uk}}

\author[Pellegrini]{\textsc{Cl\'ement Pellegrini}}
\address{Universit\'e Paul Sabatier, Toulouse 3 -- Institut de math\'ematiques de Toulouse (IMT) -- 118, route de Narbonne, 31400, Toulouse, France}
\email{{\tt clement.pellegrini@math.univ-toulouse.fr}}

\date{\today}

\allowdisplaybreaks[4]

\DeclareMathOperator{\cotan}{cotan}

\DeclareMathOperator{\eqlaw}{\stackrel{\Lc}{=}}

\def\half{\frac{1}{2}}

\def\1{{\mathbf 1}}

\def\B{{\mathbb B}}
\def\N{{\mathbb N}}
\def\Z{{\mathbb Z}}

\def\R{{\mathbb R}}

\def\H{\mathbb{H}}
\def\P{{\mathbb P}}
\def\E{{\mathbb E}}

\def\X{{\mathbb X}}

\def\Lc{{\mathcal L}}

\def\Oc{{\mathcal O}}
\def\Pc{{\mathcal P}}

\def\xb{{\mathbf x}}

\newtheorem{thm}{Theorem}[section]
\newtheorem{proposition}[thm]{Proposition}

\newtheorem{question}[thm]{Question}

\newtheorem{lemma}[thm]{Lemma}

\newtheorem{rmk}[thm]{Remark}
\newtheorem{assumption}[thm]{Assumption}

\numberwithin{equation}{section}
\numberwithin{figure}{section}

\begin{document}

\begin{abstract}
We analyze the strong noise limit of one-dimensional stochastic differential equations (SDEs). 

Our initial motivation comes from continuous measurements of open quantum systems. In this context, Bauer, Bernard and Tilloy pointed out an intriguing behavior. As the noise grows larger, the solutions exhibit locally a collapsing, that is to say, converge to pure jump processes very reminiscent of a metastability phenomenon. But surprisingly the limiting jump process is decorated by a spike process.

We give a precise meaning to the convergence and completely prove these statements for a large class of one-dimensional diffusions, thanks to a robust strategy of proof.
\end{abstract}

\renewcommand{\subjclassname}{%
  \textup{2010} Mathematics Subject Classification}

\maketitle

\setcounter{tocdepth}{3}
\medskip
\hrule
\tableofcontents
\hrule
\newpage

\section{Introduction and motivations}

We study one-dimensional diffusions
\begin{align}
\label{eq:base}
\left\{
\begin{array}{cl}
X_0^\gamma  = \ & x_0 \ , \\
dX^\gamma_t = \ & b(X^\gamma_t)dt+\sqrt{\gamma}\sigma(X^\gamma_t)dW_t \ .
\end{array}
\right.
\end{align}
where $W$ is a standard Wiener process and $b, \sigma$ are smooth functions. Throughout the paper, we assume the Itô convention. Contrary to the usual weak noise limit ($\gamma \rightarrow 0$), developed in the so-called Freidlin-Wentzell theory \cite{FW12}, we are interested in the regime where the parameter $\gamma$ goes to infinity. Our initial motivation comes from continuous quantum measurements and, as such, let us start by the guiding example which inspired us. Only then we will present the more general setting.

\medskip

{\bf Guiding example from quantum measurements:} Following Bauer, Bernard and Tilloy \cite{BD14-JP, BBT15, TBB15, BBT16}, we consider a quantum system with two energy levels $\{E_0, \ E_1\}$, a.k.a. a ``qubit'', in a thermal bath at temperature $\beta^{-1}$ and subject to continuous indirect measurements \cite{HR06, WM10} of the energy with intensity $\gamma>0$. The indirect nature of these measurements prevents the complete wave-function collapse which occurs in a direct measurement, according to the principles of quantum mechanics \cite[Section 3.6, Axiom 4]{H13}. Let us then denote by $X^\gamma_t$ the probability of measuring the energy $E_0$ at time $t$, upon a {\it hypothetical} direct measurement. The process $X^\gamma = \left( X_t^\gamma \ ; \ t \geq 0 \right)$ solves the SDE:
\begin{align}
\label{eq:BBT_SDE}
dX_t^\gamma = & -\lambda( X_t^\gamma - p ) dt + \sqrt{\gamma} X_t^\gamma \left( 1 - X_t^\gamma \right)dW_t \ .
\end{align}
In this context, the large $\gamma$ limit corresponds to the strong measurement regime. Furthermore, $\lambda>0$ is the coupling strength with the thermal bath and $p = \frac{e^{-\beta E_0}}{e^{-\beta E_0}+e^{-\beta E_1}}$ is the probability of being at the energy level $E_0$ according to a Gibbs measure. Heuristically, Eq. \eqref{eq:BBT_SDE} expresses a competition between the drift term favoring a convergence towards $p$ and the stochastic term favoring an absorption in $\{0,1\}$. In physical jargon, one says that there is a competition between thermalization and collapsing.

\begin{figure}[ht]
\includegraphics[scale=0.5]{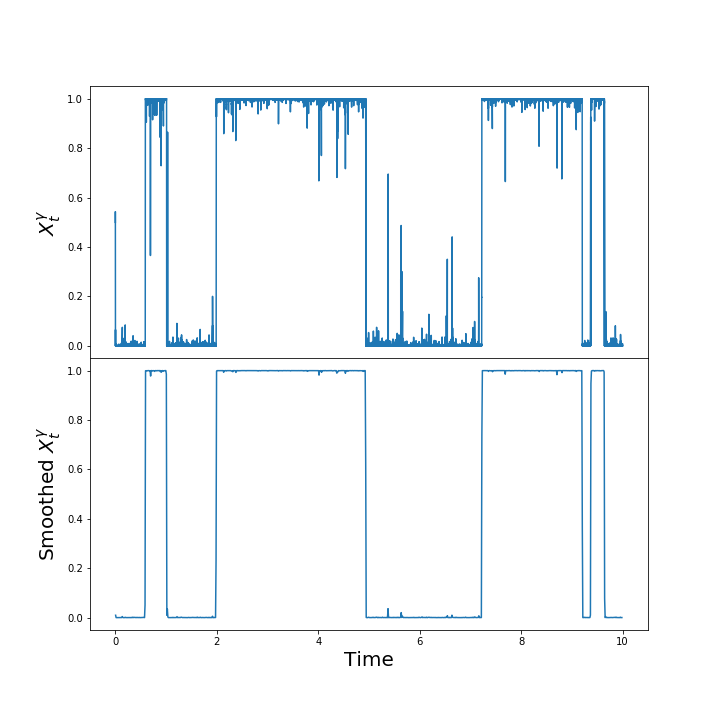}
\caption{Numerical simulation of the process $\left( X_t^\gamma; t \geq 0 \right)$ and its smoothing for $\gamma = 10^4$. Parameters are $\lambda=1.0$ and $p=0.5$. There are $10^6$ time steps. Smoothing is via averaging over 1000 steps. The code is available at the online repository 
\url{https://github.com/redachhaibi/quantumCollapse}}
\label{fig:twoState}
\end{figure}

An effective simulation at large $\gamma$ of the solution to this SDE is given in Figure \ref{fig:twoState}. From the figure, one observes:
\begin{itemize}
\item[(1)] what amounts to a jump process on $\{0,1\}$. Most of the time, $X^\gamma$ lives on a thin layer around these points where the noise vanishes.
\item[(2)] there is a decoration by spikes. These spikes are very thin since smoothing via convolution blurs them completely.
\end{itemize}

\medskip

{\bf The problem:}
Regarding the first aspect, the convergence of $X^\gamma$ to a Markov jump process has been addressed during the last years \cite{BD14-JP,BBT15,BCFFS17} and holds at the level of semi-groups. We refer to this phenomenon as a {\textit{local collapse}}. From a physical point of view, this is a metastable situation caused by the aforementioned competition between thermalization and collapsing. Let us also mention the paper \cite{KL18}, where the authors prove for specific SDEs the approximation by a Markov jump process via the study of hitting times and their asymptotics in $\gamma \rightarrow \infty$.

The second aspect is more surprising and much less understood. It was first described in \cite{TBB15,BBT16} and then studied in greater depth in \cite{BB18}. In fact, fluctuations around the local collapse {\it do} persist in the strong noise limit and take the form of ``spikes'' decorating the Markov jump process. So far, there is only a limited understanding of the convergence topology and the precise statistics of these spikes.

A general approach developed in \cite{BBT16, BB18} concerns a change of time (a zooming) which allows to consider the presence of spikes. More precisely, the spikes are explained in terms of excursions of a reflected Brownian motion which appears in the strong noise limit. In particular, in order to obtain their result, the authors of \cite{BB18} prove an effective approximate version of the Skorohod Lemma. 

From this body of literature arise the following questions:
\begin{question}
\ 
\begin{itemize}
\item How to formalize the limiting phenomena? At this point, the precise mathematical nature of the "spike process" is unclear. Even the name "process" is unwarranted for now.
\item Is there a limit theorem for the process $X^\gamma$ as $\gamma \rightarrow \infty$? A satisfying answer should be two-fold. On the one hand, we need the convergence to a Markov jump process, which holds only upon smoothing. On the other hand, the convergence to the spike process needs to happen in a non-standard topology, which we need to describe.
\end{itemize}
\end{question}

In the general setup of \eqref{eq:base}, the goal of this paper is three-fold 
\begin{itemize}
\item we provide a precise statement for the convergence process $X^\gamma$, so that both the Markov jump process and the spike process are obtained in the limit.
\item we formalize the convergence topology as a Hausdorff convergence of graphs, in order to capture spikes.
\item we give a parsimonious description of the law of the limiting processes.
\end{itemize}

\medskip

{\bf The results:} 
The previous papers \cite{BB18, KL18} perform a rigorous study only after an uncontrolled perturbative analysis around one of the stable points i.e. points where the noise vanishes. Thus such papers restrict themselves to SDEs living in $[0,\infty)$ with particular coefficients which satisfy
\begin{equation*}
\quad \sigma(x)>0 \quad {\rm{for\;  all\; }} x\in (0,\infty), \quad \sigma (0)=0, \quad b (0)>0.
\end{equation*}
This does not cover the two-boundary case such as Eq. \eqref{eq:BBT_SDE}, relevant for quantum mechanics. Moreover, a precise statement describing the spikes has been missing.

In the present paper, we do not perform any approximation, treat generic coefficients and give a precise description of the spike process. We provide a general technique to study the strong noise limit $\gamma \to \infty$ of one-dimensional SDEs with two possible setups. In the first half of the paper, we have the following working hypotheses:
\begin{assumption}
\label{assumption:main}
We assume that the  drift term $b$ and the diffusion coefficient $\sigma$ are Lipschitz continuous so that the SDE \eqref{eq:base} admits strong solutions. Moreover
\begin{align*}
& \sigma(x)>0 \quad {\rm{for\;  all\; }} x\in (0,1), \quad \sigma (0)=\sigma(1)=0 \ ,\\
& b (0)>0,\quad  b (1)<0 \ .
\nonumber
\end{align*}

Recalling that $x_0$ denotes the initial position, we naturally consider $x_0 \in (0,1)$: the starting point needs to be between points where the noise vanishes.
\end{assumption}

In fact, such assumptions can be slightly relaxed, as long as the SDE continues to have strong solutions and that the points $\{0,1\}$ are natural boundaries in the sense of Gilman-Skorohod - see \cite[Section 6.9]{K05} for a definition and references. However, our method of proof still needs $b(0)>0$ and $b(1)<0$.

Our main results in this setup are provided in Theorem \ref{thm:limit_thm_1d}. It shows first the convergence of the process $(X_t^\gamma\, ;\, t\ge 0)$ to a jump Markov process $(\xb_t\, ;\, t\ge 0)$ as $\gamma \to \infty$.  A reader used to problems of weak convergence of stochastic processes will notice that the previous convergence cannot hold in the usual Skorohod topology since $(X_t^\gamma\, ; \, t\ge 0)$ has continuous paths while $(\xb_t\, ;\, t\ge 0)$ has only c\`adl\`ag trajectories. The statement holds only upon smoothing, which is equivalent to the convergence of semi-groups. Hence the precise statement is that for every compactly supported continuous function $f$ of time and space
 $$\lim_{\gamma \to \infty} \int_0^\infty f(t, X_t^\gamma) \,  dt \; =\; \int_0^\infty f(t, \xb_t)\,  dt \quad \P - {\rm{a.s.}}$$
Almost sure convergence is due to a particular coupling of $X^\gamma$ for different $\gamma$. 
 
The previous convergence does not detect the spikes that are observed in the numerical simulation given in Fig. \ref{fig:twoState}. Therefore, in order to mathematically capture them, we have to find the right topology. Our solution uses the Hausdorff metric on the graphs of functions. Indeed the second part of our theorem establishes the convergence of $(X_t^\gamma \, ;\, t\ge 0)$ to a spike process defined in terms of excursions. Within this approach, we obtain the complete picture of Markov jump processes with spikes and we make precise the statistics of the involved processes. 

\medskip

In the second half of the paper, we chose another setup focusing on an important equation also coming from quantum continuous measurements. This setup is related to the measurement of Rabi oscillations, where the drift $b$ depends on the parameter $\gamma$ and where the boundaries have different behaviors. In the large $\gamma$ regime, markedly different limiting processes appear compared the two boundary case covered by Assumptions \ref{assumption:main}. This example is rich enough to illustrate the robustness of our approach, without attempting a complete description of large noise limits for all one-dimensional diffusions. 

\section{Main setup and associated statements}

\subsection{Two limiting processes}
We start by defining the two processes which shall appear in the main theorem.

On the one hand, we define $\left( \xb_t \ ; \ t \geq 0 \right)$ as the continuous time (càdlàg) Markov process with state space $\{0, 1\}$ and jump rates $W$:
$$ W^{0,1} = \left|b(0)\right| , \quad \quad W^{1,0} = \left|b(1)\right|\ .$$
Here, we wrote $W^{i,j}$ for the jumping rate from state $i$ to state $j$. The initial position is sampled according to
$$ \P\left( \xb_0 = 1 \right) = 1-\P\left( \xb_0 = 0 \right) = x_0 \ .$$

\begin{figure}[ht]
\includegraphics[scale=0.4]{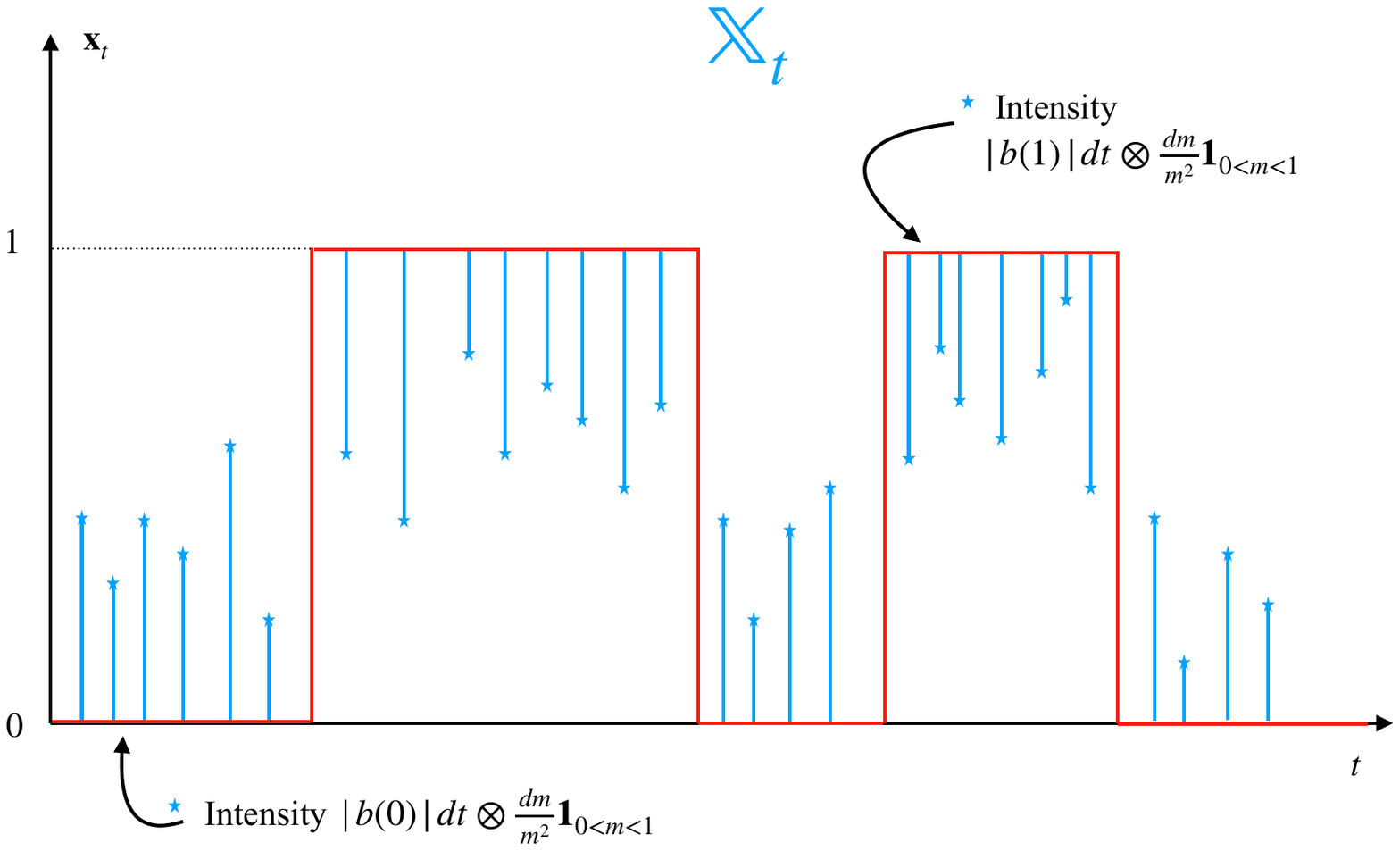}
\caption{Sketch of the two limiting processes. The Markov pure jump process $\xb$ is in red, and the set-valued spike process $\X$ is in blue. }
\label{fig:sketch}
\end{figure}

On the other hand, we define the spike process as a {\it set-valued} random path $\X: \R_+ \rightarrow \Pc\left( [0,1] \right)$,  where $\Pc\left( [0,1] \right)$ is the power set of the segment $[0,1]$. For a comprehensive sketch, see Figure \ref{fig:sketch}. It is formally obtained as follows:
\begin{itemize}
\item Sample a random initial segment $\X_0$ as
$$
   \X_0 = \left\{
\begin{array}{ll}

[Y,1] \textrm{ when $\xb_0 = 1$,} & \P\left( Y \in dy \ | \ \xb_0 = 1\right) = \frac{1-x_0}{x_0} \mathds{1}_{\{0 < y < x_0\}} \frac{dy}{(1-y)^2} \ ,\\

[0,Y] \textrm{ when $\xb_0 = 0$,} & \P\left( Y \in dy \ | \ \xb_0 = 0\right) = \frac{x_0}{1-x_0} \mathds{1}_{\{x_0 < y < 1\}} \frac{dy}{y^2} \ .
\end{array}
   \right.
$$
\item Sample $\left( t, \widetilde{M}_t \right)$ following a Poisson point process on $\R_+ \times [0,1]$ with intensity
$$ \left( dt \otimes \frac{dm}{m^2} \mathds{1}_{\{ 0 \leq m < 1 \}} \right) \ .$$
Then, by progressively rescaling time for $\left( t, \widetilde{M}_t \right)$ by
$$
   \left\{
\begin{array}{ccc}
\ \left|b(0)\right|^{-1} & \textrm{ when } & \xb_t = 0 \ ,\\
\ \left|b(1)\right|^{-1} & \textrm{ when } & \xb_t = 1 \ ,\
\end{array}
   \right. \
$$
we obtain a Poisson point process with random intensity which we denote by $\left(t, M_t \right)$. 

Equivalently, one can sample two independent Poisson point processes on $\R_+ \times [0,1]$ with the above intensity, and then rescale them in time by $\left|b(0)\right|^{-1}$ and $\left|b(1)\right|^{-1}$ respectively. The process $\left(t, M_t \right)$ can also obtained by picking either the first or the second one, depending on the current value of $\xb$.
\item Finally
$$
   \X_t = \left\{
\begin{array}{ccc}
\ [0, M_t]   & \textrm{ if } & \xb_t = \xb_{t^-} = 0 \ ,\\
\ [1-M_t, 1] & \textrm{ if } & \xb_t = \xb_{t^-} = 1 \ ,\\
\ [0, 1]     & \textrm{ if } & \xb_t \neq \xb_{t^-} \ .
\end{array}
   \right.
$$
\end{itemize}

Notice that by virtue of $(t, M_t)$ being a Poisson point process with finite intensity away from zero, there are no points with the same abscissa and only countably many $t \in \R_+$ with $M_t>0$. If there is no point with abscissa $t \in \R_+$, then it is natural to set $M_t = 0$ and thus $\X_t = \{\xb_t\}$. This convention accounts for the infinite measure at zero.

\subsection{Main result} We can now state:
\begin{thm}[Main Theorem]
\label{thm:limit_thm_1d}
Under Assumptions \ref{assumption:main}, it is possible to couple the processes $\left( \xb, \X \right)$ and $X^\gamma$ for all values of $\gamma>0$ on the same probability space, so that the following limits hold almost surely.

\begin{itemize}
\item Upon smoothing via a continuous function with compact support $f: \R_+ \times \R \rightarrow \R$, we have the almost sure convergence:
\begin{align}
\label{eq:limit_process_xb}
   \lim_{\gamma \rightarrow \infty}
   \int_0^\infty f( t, X_t^\gamma) \, dt
   = &
   \int_0^\infty f( t, \xb_t ) \, dt
   \ .
\end{align}
\item In the sense of {\it Hausdorff convergence of closed sets}, for all $H>0$, we have the almost sure convergence of graphs:
\begin{align}
\label{eq:limit_process_X}
\lim_{\gamma \rightarrow \infty}
\left( X_t^\gamma ; \ 0 \leq t \leq H \right) 
= & \left( 
	\X_t
   ; \ 0 \leq t \leq H
   \right)
   \ .
\end{align}
\end{itemize}
\end{thm}

\begin{rmk}[Explanations]
\label{rmk:limit_thm_1d_explanation}
The first part of the theorem can be loosely reformulated by saying that the convergence of $X^\gamma$ to $\xb$ holds upon smoothing, which amounts to deleting the spikes. Nevertheless, one needs an appropriate notion of convergence in order to capture the spikes, which are infinitely thin in the limit. Thus, we resort to the Hausdorff metric on the collection of closed sets of $[0,H] \times [0,1]$ \cite[Ex 7, p.280]{Munkres2000}. It is defined on closed sets $A$ and $B$ via:
\begin{align}
\label{def:hausdorff}
   {\rm d}_{\H}(A,B)
 := & \ \inf\left\{ \varepsilon>0 \ |
               \ A \subset B + \varepsilon \B \ ,
               \ B \subset A + \varepsilon \B
               \right\} \ ,
\end{align}
where $\B$ is the unit ball. The second part of theorem says that ${\rm d}_{\H}( A^\gamma, B ) \rightarrow 0$ where $A^\gamma$ is the graph of $X^\gamma$:
$$ A^\gamma := \left\{ (t, X_t^\gamma) \ , \ 0 \leq t \leq H \right\}
             = \bigsqcup_{ 0 \leq t \leq H} \left\{ t \right\} \times \{ X_t^\gamma \} \ $$
and $B$ is given by
$$ B := \left\{ (t, x) \ , \ 0 \leq t \leq H, \ x \in \X_t \right\} = \bigsqcup_{ 0 \leq t \leq H} \left\{ t \right\} \times \X_t \ . $$
The set $B$ is seen as the graph of the multi-valued function $\X$. The fact that $B$ is a closed set comes as a by-product of the proof.
\end{rmk}

Note that the dynamics of the jump process $\xb$ and of the spike process $\X$ depend on the characteristics $(b, \sigma^2)$ of the initial diffusion \eqref{eq:base} only through the absolute value of the drift $b$ at 0 and 1. Indeed, as long as Assumption \ref{assumption:main} is satisfied, these limiting processes are identical no matter the values of the drift $b$ and the diffusivity $\sigma$ in the bulk $(0,1)$. In fact, the only impact of $\sigma$ is in the selection of the space where the jump process $\xb$ will live, i.e. here $\{0, 1\}$.

This can be contrasted with the jump processes appearing in the weak noise limits \cite[Chapter 6]{FW12}. In this case, with diffusions whose drift is the gradient of a potential, the limiting jump processes depend on the full landscape given by this potential. Therefore, it is fair to say that strong noise limits are far more universal than weak noise limits.

\subsection{Further remarks}
Our approach uses crucially scale functions and the Dambis-Dubins-Schwarz Theorem, which are one-dimensional tools. In particular, extensions to a multi-dimensional setting are absolutely not straightforward. This is a current subject of investigations.

On a side note, the coupling of the processes $\left(\xb, \X \right)$ with $X^\gamma$ for different $\gamma>0$ is nothing but a convenient device. Exactly like Skorohod's Representation Theorem \cite[Theorem 6.7]{B13}, it allows to recast weak convergence to an almost sure convergence. Such a coupling is particularly convenient in order to avoid formalizing the weak convergence of random closed sets in the Hausdorff topology.

Finally, let us comment on the choice of the Hausdorff topology. It is very natural as it corresponds to the intuition gained from observing the plot in Figure \ref{fig:twoState}. This topology is not so uncommon if we were interested in the convergence in the space $D$ of càdlàg functions: in this case, Hausdorff convergence of completed graphs is the Skhorohod $M_2$ topology. See \cite[Section 11.5]{W02} for a definition and the relationship to the other Skorohod topologies on $D$. Because the limiting process $\X$ does not belong to $D$, we cannot invoke the previous literature and we are thus forced to give an independent treatment.

\section{Proof of the Main Theorem \ref{thm:limit_thm_1d}}

As mentioned before, the main ideas behind the proof of Theorem \ref{thm:limit_thm_1d} are very specific to dimension $1$:
\begin{itemize}
\item Every one-dimensional diffusion is a Brownian motion upon changing space and time. 
\item And when time is parametrized by the inverse of local time, a Brownian trajectory can be broken into excursions thanks to Itô's Excursion Theory.
\end{itemize}

More precisely, the first point is achieved by sucessively composing with the scale function in order to obtain a martingale and then invoking the Dambis-Dubins-Schwarz Theorem \cite[Chapter V, Theorem 1.6]{RY}. The details are given in Subsection \ref{subsection:rescaling_space_time}. There, we make explicit the coupling for different $\gamma>0$. After computing the asymptotics of these changes of scale in Subsection \ref{subsection:asymptotics}, we are able to force the appearance of the limiting processes $\xb$ and $\X$. The convergences to $\xb$ and $\X$ are respectively treated in Subsections \ref{subsection:limiting_xb} and \ref{subsection:limiting_X}. Only the construction of $\X$ will require Itô's Excursion Theory \cite[Chapter XII]{RY}.

\subsection{Additional notation}
If $X = \left( X_t \ ; \ t \geq 0 \right)$ is a real-valued process, then the hitting time of $a \in \R$ is
$$ \tau_a(X) := \inf\{ t \geq 0 \ | \ X_t \geq a \} \ .$$
When the underlying process is understood from context, we write $\tau_a = \tau_a(X)$.

Also, if $L: t \in \R_+ \mapsto L_t \in \R_+$ is non-decreasing, then its left-inverse is the function defined by 
$$ \forall \ell \ge 0, \quad L_\ell^{\langle -1 \rangle} := \inf\left\{ t \geq 0 \ | \ L_t \geq \ell \right\} \ .$$

\subsection{Rescaling space and time}
\label{subsection:rescaling_space_time}
\ 
\medskip

\begin{itemize}

\item {\bf Scale function:}
Consider the scale function $h_\gamma$ which is harmonic for the process $\left( X_t^\gamma \ ; \ t \geq 0 \right)$. It solves for $0<x<1$:
$$ b(x) h_{\gamma}^\prime (x) + \frac{\gamma}{2} \sigma^2(x) h_{\gamma}^{\prime\prime}(x) = 0,$$
or equivalently
$$ \frac{h_{\gamma}^{\prime\prime}(x)}{h_{\gamma}^{\prime}(x)}
 = -\frac{2 b(x)}{\gamma \sigma^2(x)} .$$
Since $h_\gamma$ is unique up to affine transformations, we choose $h_\gamma(x_0)=x_0$ and $h_\gamma'(x_0)=1$, which leads to:
\begin{align}
\label{def:scale}
   h_\gamma(x) := x_0 + \int_{x_0}^x \exp
   \left( -\int_{x_0}^{y} \frac{2 b(z)}{\gamma \sigma^2(z)} dz \right) dy \ .
\end{align}

\medskip

\item {\bf Time change:}
As announced, we invoke the Dambis-Dubins-Schwarz Theorem in order to write
\begin{align}
\label{eq:coupling}
 h_\gamma (X_t^\gamma) & = \beta_{T_t}
\end{align}
where $\beta$ is a Brownian motion starting at $x_0$ and 
$$ T_t :=T_t^\gamma = \gamma \int_0^t (h_{\gamma}^\prime (X_s^\gamma))^2\,  \sigma^2(X_s^\gamma) \ ds \ .$$
At this point, we need to mention that the coupling for different values $\gamma>0$ takes place in Eq. \eqref{eq:coupling}: we start with sampling the Brownian motion $\beta$, and only then the process $X^\gamma$ is defined through that equation for different values of $\gamma>0$. Taking the inverse, we get 
$$ dT^{\langle -1 \rangle}_{\ell}
 = \frac{d \ell}
        {\gamma \left[h_\gamma^\prime \left(X^\gamma_{T^{{\langle -1 \rangle}}_{\ell} } \right) \right]^2
\,  \sigma^2\left( X^\gamma_{T^{{\langle -1 \rangle}}_{\ell} } \right)} \ .
$$
Since 
$$X^\gamma_{T^{{\langle -1 \rangle}}_{\ell} } = h_\gamma^{\langle -1 \rangle} ( \beta_{\ell}),$$
we get 
$$ dT^{{\langle -1 \rangle}}_{\ell}
 = \frac{d \ell}
        {\gamma \big[ (h_\gamma^{\prime}  \circ h_\gamma^{{\langle -1 \rangle}}) (\beta_{\ell})\big]^2
        \, \left( \sigma^2 \circ h_\gamma^{{\langle -1 \rangle}} \right)(\beta_{\ell}) }
 =: \varphi_\gamma(\beta_{\ell}) d \ell \ .$$

In the end, we have that:
$$ X_t^\gamma = h_\gamma^{{\langle -1 \rangle}} (\beta_{T_t}) \ ,$$
where $T_t$ can be defined by 
$$\int_0^{T_t} \varphi_\gamma(\beta_{\ell}) d \ell = t \ .$$

\end{itemize}

\subsection{Asymptotics for the changes of scale}
\label{subsection:asymptotics}

\begin{lemma}
\label{lemma:asymptotics}
On the one hand, we have the following convergence, uniformly in $y \in \R$:
\begin{align}
\label{eq:cv_h}
 h_\infty^{\langle -1 \rangle}(y) & 
:= \lim_{\gamma \rightarrow \infty} h_\gamma^{\langle -1 \rangle}(y)
 = \left\{
   \begin{array}{ccc}
   0 & \textrm{ if } & y \leq 0       \\
   y & \textrm{ if } & 0 \leq y \leq 1\\
   1 & \textrm{ if } & 1 \leq y
   \end{array}
   \right.
   \ .
\end{align}

And on the other hand, we have the weak convergence:
\begin{align}
\label{eq:cv_varphi}
   \varphi_\gamma & 
   \stackrel{\gamma \rightarrow \infty}{ \longrightarrow }
   \frac{1}{2 |b(0)|} \delta_0 + \frac{1}{2 |b(1)|} \delta_1 \ .
\end{align}
\end{lemma}

\begin{proof}
Thanks to the Asummption \ref{assumption:main} on $b$ and $\sigma$, we see that:
$$ \lim_{x \rightarrow 0^+} h_\gamma(x) = -\infty \ ,
   \textrm{ and }
   \lim_{x \rightarrow 1^-} h_\gamma(x) = \infty  \ .$$
Adding to that the expression \eqref{def:scale}, we see that $h_\gamma$ is a strictly increasing diffeomorphism from $(0,1)$ to $\R$. Moreover, $h_\gamma$ as $\gamma \rightarrow \infty$ tends uniformly to the identity on $(\varepsilon, 1-  \varepsilon)$ for any fixed $\varepsilon \in (0,\half)$. The pointwise convergence in Eq. \eqref{eq:cv_h} thus follows. The result is strengthened to uniform convergence using Dini's Theorem. In fact, the expression $h_\infty^{\langle -1 \rangle}$ should just be understood as a convenient notation.  It is not by any means the inverse of a real valued function.

Now, let us focus on the proof of Eq. \eqref{eq:cv_varphi}. Let $f: \R \rightarrow \R$ be a continuous compactly supported function. Using the change of variables $y = h_\gamma(x)$, we have:
 \begin{align*}
      \int_\R f (y) \varphi_\gamma (y) \, dy 
  = & \int_\R \frac{f (y) }
                   {\gamma \ 
                    (\sigma^2 \circ h_\gamma^{\langle -1 \rangle}) (y)
                    \left(
					(h_\gamma' \circ h_\gamma^{{\langle -1 \rangle}})
					\right)^{2}  (y)
				    } \, dy \\
  = & \int_0^1 \frac{(f \circ h_\gamma) (x)}
                    {\gamma \sigma^2(x)
		       h'_\gamma(x)}\, dx \\
  = & \int_0^1 \frac{(f \circ h_\gamma) (x)}
                    {\gamma \sigma^2(x) }
                      \exp\left( \int_{x_0}^{x} \frac{2 b(z)}{\gamma \sigma^2(z)} dz \right)\, dx \ .
 \end{align*}
 Because the integrand is ill-behaved only in the neighborhood of $0$ and $1$, we can discard the middle range. More precisely, let $\varepsilon>0$ be a small constant such that each of $b_{|[0, \varepsilon]}$ and $b_{|[1-\varepsilon, 1]}$ remain with constant sign, bounded away from zero. We can write:
 \begin{align*}
      \int_\R f (y) \varphi_\gamma (y) \, dy
  = & \ 
      o_\gamma(1)
      +
      \int_{(0,\varepsilon] \sqcup [1-\varepsilon, 1)}
      \frac{(f \circ h_\gamma) (x)}
           {\gamma \sigma^2(x) }
           \exp\left( \int_{x_0}^{x} \frac{2 b(z)}{\gamma \sigma^2(z)} dz \right)\, dx \ .
 \end{align*}
 
Now notice that $\frac{1}{h_\gamma'}: x \mapsto \exp\left( \int_{x_0}^{x} \frac{2 b(z)}{\gamma \sigma^2(z)} dz \right)$ is increasing on $(0,\varepsilon]$ and decreasing on $[1-\varepsilon, 1)$. Also, the limits at the boundary are $\frac{1}{h_\gamma'}(0^+)=\frac{1}{h_\gamma'}(1^-)=0$. As such, by performing a change of variable:
 \begin{align*}
      &\int_\R f (y) \varphi_\gamma (y) \, dy\\
  = & \ o_\gamma(1) + 
      \int_{{(0,\varepsilon] \sqcup [1-\varepsilon, 1)}}
      (f \circ h_\gamma) (x) 
                  \frac{1 }
                       {2 b(x) }
                  \cfrac{d}{dx}\left[ \exp\left( \int_{x_0}^{x} \frac{2 b(z)}{\gamma \sigma^2(z)} dz \right) \right] \, dx\\
  = & \ o_\gamma(1) + 
      \int_0^{\frac{1}{h_\gamma'(\varepsilon)}  } \frac{(f \circ h_\gamma \circ q^1_\gamma)  (r)} {2 ( b \circ q^1_\gamma) (r) } \, dr
    - \int_0^{\frac{1}{h_\gamma'(1-\varepsilon)}} \frac{(f \circ h_\gamma \circ q^2_\gamma)  (r)} {2 ( b \circ q^2_\gamma) (r) } \, dr
      \ ,
 \end{align*}
 where $q^1_\gamma = \left( \frac{1}{h_\gamma'}_{\big\vert(0,\varepsilon]  } \right)^{\langle -1 \rangle}$
   and $q^2_\gamma = \left( \frac{1}{h_\gamma'}_{\big\vert[1-\varepsilon,1)} \right)^{\langle -1 \rangle}$. We are done upon applying Lebesgue's Dominated Convergence Theorem to the following limits. On the one hand $q^1_\gamma \rightarrow 0$ and $q^2_\gamma \rightarrow 1$ as $\gamma \rightarrow \infty$, since $\frac{1}{h_\gamma'}$ converges to vertical lines at zero and $1$. On the other hand, given the simple convergence: 
$$ \left(h_\gamma^{\langle -1 \rangle}\right)'
   \stackrel{\gamma \rightarrow \infty}{\longrightarrow}
   \mathds{1}_{[0,1]} \ ,
$$
taking left-inverses for the ascending and descending parts yields:
$$ h_\gamma \circ q^1_\gamma
 = \left( \frac{1}{{h_\gamma'}_{\big|(0,\varepsilon]}} 
          \circ h_\gamma^{\langle -1 \rangle} \right)^{\langle -1 \rangle}
 = \left( \left(h_\gamma^{\langle -1 \rangle}\right)_{\big|(-\infty,h_\gamma(\varepsilon)]}' \right)^{\langle -1 \rangle}
 \stackrel{\gamma \rightarrow \infty}{\longrightarrow} 0 \ ,
$$
$$ h_\gamma \circ q^2_\gamma
 = \left( \frac{1}{{h_\gamma'}_{\big|[1-\varepsilon,1]}}
   \circ h_\gamma^{\langle -1 \rangle} \right)^{\langle -1 \rangle}
 = \left( \left(h_\gamma^{\langle -1 \rangle}\right)_{\big|[h_\gamma(1-\varepsilon), \infty)}' \right)^{\langle -1 \rangle}
 \stackrel{\gamma \rightarrow \infty}{\longrightarrow} 1 \ .
$$
We have thus proved that:
$$
   \int_\R f (y) \varphi_\gamma (y)\, dy \stackrel{\gamma \rightarrow \infty}{\longrightarrow}
   \frac{f(0)}{2\left| b(0) \right|}
   +
   \frac{f(1)}{2\left| b(1) \right|}
   \ .   
$$
\end{proof}

\subsection{The limiting process \texorpdfstring{$\xb$}{xb} and proof of Eq. \texorpdfstring{\eqref{eq:limit_process_xb}}{}}
\label{subsection:limiting_xb}

The natural time scale for the process $X^\gamma$ will be referred as {\it real time}. And the changed scale, which is natural for the DDS Brownian motion $\beta$, will be referred to as {\it effective time}. This follows the denomination of Bauer, Bernard and Tilloy; and it is helpful in explaining proofs where several time scales interact.

We start by a simple yet crucial lemma:
\begin{lemma}
Let $L^a(\beta)$ be the local time accumulated by $\beta$ at the point $a$. We have the almost sure convergence:
\begin{align}
\label{eq:cv_T_inv}
 T^{\langle -1 \rangle}_{\ell} \stackrel{\gamma \rightarrow \infty}{\longrightarrow} 
   & 
   \frac{1}{2|b(0)|} L^0_\ell(\beta) 
 + \frac{1}{2|b(1)|} L^1_\ell(\beta) 
   \ ,
\end{align}
uniformly on all compact sets of the form $[0,L]$.
\end{lemma}
\begin{proof}
Thanks to the occupation time formula \cite[Chapter VI, Corollary 1.6]{RY}, we have:
$$ T^{\langle -1 \rangle}_{\ell}
 = \int_0^\ell \varphi_\gamma( \beta_u ) \, du
 = \int_\R \varphi_\gamma(a) L_\ell^a(\beta) \, da \ .$$
As customary, we are considering a version of local time so that the map $a \mapsto L^a_\ell(\beta)$ is continuous and compactly supported. Therefore, the previous Lemma \ref{lemma:asymptotics} immediately yields the pointwise convergence in Eq. \eqref{eq:cv_T_inv}. This convergence holds uniformly in $\ell \in [0, L]$, thanks to Dini's Theorem. Notice that by doing so, we avoid invoking ``the approximate Skorohod Reflection Theorem'' from \cite{BBT16} which gives local times in the approach of Bauer, Bernard and Tilloy.
\end{proof}

This shows the importance of the process $\left( \sigma_t; t \geq 0 \right)$ defined as:
$$ \sigma_t
   :=
   \inf \left\{ \ell \geq 0, \ 
   \frac{L_\ell^0(\beta)}{2|b(0)|} + \frac{L_\ell^1(\beta)}{2 |b(1)|} > t \right\}
   .
$$
It is a pure jump process defined as the inverse of mixed local times accumulated by $\beta$ at levels $0$ and $1$. Clearly, it will be the time change of interest from real time to effective time. Already, we see that $\sigma$ is a key ingredient in the following proposition, which shows the appearance of the pure jump process $\xb$.

\begin{proposition}
\label{proposition:equivalent_MC}
Let $\beta$ be a standard one-dimensional Brownian motion started at $x_0 \in [0,1]$. The process
$$(\xb_t \, ;\, t\ge 0) := (\beta_{\sigma_t}\, ; \, t\ge 0)$$
is a càdlàg $\{0,1\}$-valued Markov process with $\P\left( \xb_0 = 1 \right) = x_0$ and jump rates $W$ where:
$$ W^{0,1} = |b(0)| \ ,
\quad
   W^{1,0} = |b(1)| \ .$$
\end{proposition}
\begin{proof}
Firstly, given that $\sigma$ is càdlàg with jumps corresponding to excursions of $\beta$ away from $0$ and $1$, $\xb_t = \beta_{\sigma_t}$ is indeed càdlàg $\{0,1\}$-valued. Moreover, since $(\sigma_t\, ; \, t\ge 0)$ is an increasing collection of stopping times, $(\xb_t\, ;\, t\ge 0)$ is Markovian - possibly not homogenous in time.

Secondly, $\P(\xb_0=1)= \P_{x_0} (\beta_{\sigma_0} =1)= \P_{x_0} ( \tau_1 (\beta) < \tau_0 (\beta)) =: s(x_0)$ where $\tau_a (\beta)$ is the hitting time of $a$ by $\beta$. Since $s$ is a harmonic function for $\beta$ with boundary conditions $s(1)=1$, $s(0)=0$, we deduce that $s(x_0)=x_0$. Finally, let us prove that $\xb$ is Markov with jump rates as described. Suppose we run the Brownian motion $\beta$ started from $0$ and killed upon hitting $1$. Let $\tau_1(\beta)$ be first time that $\beta$ reaches $1$ - this is effective time. And let $\tau_1(\beta_\sigma)$ be first time that $\beta_{\sigma_.}$ reaches $1$ - this is real time. Given the definition of $\sigma$, we have $\frac{L_{\tau_1(\beta)}^0(\beta)}{2|b(0)|} = \tau_1( \beta_\sigma )$. Thanks to the first Ray-Knight Theorem \cite[Chapter XI, Theorem 2.2]{RY}, the accumulated local time $L_{\tau_1(\beta)}^0(\beta)$ is the square of a two dimensional Bessel process at time $1$, or equivalently $2{\mathcal E} (1)$ where $\mathcal E (1)$ is a standard exponential random variable. As such $\tau_1(\beta_\sigma) = \frac{2 {\mathcal E} (1) }{2|b(0)|} = \frac{{\mathcal E} (1) }{|b(0)|}$ and hence the jumping rate from $0$ to $1$ is $W^{0,1} = |b(0)|$.

The proof that $W^{1,0} = |b(1)|$ follows exactly the same lines.
\end{proof}

We can now finish the proof of Eq. \eqref{eq:limit_process_xb}. Recall that 
$$ X_t^\gamma
 = (h_\gamma^{\langle -1 \rangle} \circ h_\gamma) ( X_t^\gamma )
 = h_\gamma^{\langle -1 \rangle}( \beta_{T_t} ) \ .
$$

Taking $f$ compactly supported and performing a change of variable:
\begin{align*}
    \int_0^\infty f\left( t, X_t^\gamma \right) dt
= & \int_0^\infty f\left( t, h_\gamma^{\langle -1 \rangle}( \beta_{T_t} ) \right) dt\\
= & \int_0^\infty f\left( T^{\langle -1 \rangle}_{\ell}, h_\gamma^{\langle -1 \rangle}( \beta_\ell ) \right)
                  d T^{\langle -1 \rangle}_{\ell} \ .
\end{align*}
Because of Eq. \eqref{eq:cv_T_inv}, we have the weak convergence $dT^{\langle -1 \rangle}_{\ell} \stackrel{\gamma \rightarrow \infty}{\longrightarrow} d\sigma^{\langle -1 \rangle}_\ell$. Adding to that the convergence of the inverse scale function $h_\gamma^{\langle -1 \rangle}$, we have:
\begin{align*}
    \int_0^\infty f(t, X_t^\gamma) dt
\stackrel{\gamma \rightarrow \infty}{\longrightarrow}
    & \int_0^\infty \mathds{1}_{\beta_\ell \in [0,1]} f ( \sigma^{\langle -1 \rangle}_\ell, \beta_\ell ) d \sigma^{\langle -1 \rangle}_\ell \\
  = & \int_0^\infty \mathds{1}_{\beta_{\sigma_t} \in [0,1]} f ( t, \beta_{\sigma_t} ) dt \\
  = & \int_0^\infty f ( t, \xb_t ) dt \ .
\end{align*}
This concludes the proof of the first part of the theorem. Notice that the above weak convergence argument does blur the fine properties of the spike process $\X$, which we now move on to.

\subsection{The limiting process \texorpdfstring{$\X$}{X} and proof of Eq. \texorpdfstring{\eqref{eq:limit_process_X}}{}}
\label{subsection:limiting_X}

We shall use the notation from the explanatory Remark \ref{rmk:limit_thm_1d_explanation}. As such let $A^\gamma$ be the graph of $X^\gamma$. For $\gamma>0$, we have via time change:
\begin{align*}
A^\gamma
  & = \left\{ (t, X_t^\gamma) \ | \ 0 \leq t \leq H \right\} \\
  & = \left\{ \left( T^{\langle -1 \rangle}_\ell, h^{\langle -1 \rangle}_\gamma(\beta_\ell) \right) \ | \ 0 \leq \ell \leq T_H \right\} \\
  & \stackrel{\gamma \rightarrow \infty}{\longrightarrow}
      \left\{ \left( \sigma^{\langle -1 \rangle}_\ell, h_\infty^{\langle -1 \rangle}(\beta_\ell) \right) \ | \ 0 \leq \ell \leq \sigma_H \right\},
\end{align*}
the limit holding in the Hausdorff topology. Indeed, one can easily show that uniform convergence of maps yields the Hausdorff convergence of their images. Moreover, $A^\gamma$ is the image of the the map $\ell \mapsto \left( T^{\langle -1 \rangle}_\ell, h_\gamma^{\langle -1 \rangle}(\beta_\ell) \right)$ on the (random) interval $[0, T_H]$ and that sequence of maps converges uniformly to $\ell \mapsto \left( \sigma^{\langle -1 \rangle}_\ell, h_\infty^{\langle -1 \rangle}(\beta_\ell) \right)$. 

All that remains is proving that 
$$ \left\{ \left( \sigma^{\langle -1 \rangle}_\ell, h_\infty^{\langle -1 \rangle}(\beta_\ell) \right)
           \ | \ 0 \leq \ell \leq \sigma_H
   \right\}
 = \bigsqcup_{t \in [0,H]} \left\{ t \right\} \times B_t
$$
is as in the statement of theorem. Equivalently, we need to prove that for every $t \in [0,H]$, the set $\X_t$ can be obtained as
\begin{align*} 
B_t = & \left\{ \ h_\infty^{\langle -1 \rangle}(\beta_\ell) \ | \ \sigma_\ell^{\langle -1 \rangle} = t \right\} \\
    = & \left\{ \ \beta_\ell \mathds{1}_{\{ 0 \leq \beta_\ell \leq 1 \}}
                  + \mathds{1}_{\{ \beta_\ell > 1 \}} \ \big| \ \frac{L_\ell^0(\beta)}{2|b(0)|} + \frac{L_\ell^1(\beta)}{2 |b(1)|} = t \right\} \ ,
\end{align*}
which is given by the following proposition. It gives a very efficient simulation scheme for the spike process. 

\begin{proposition}
In order to sample the limiting aforementioned random set
$$ \bigsqcup_{t \in [0,H]} \left\{ t \right\} \times B_t $$
with $X_0^\gamma = x_0 \in [0,1]$, one has to:

\begin{enumerate}
 \item {\bf Simulation of the equivalent $\{0,1\}$ Markov chain:}
 
 Run $\left( \xb_t, \ t\geq 0 \right)$ started at $x_0$, as in Proposition \ref{proposition:equivalent_MC}. 
 
 \item {\bf Simulation of the first spike $B_0$:}

 If $x_0 = 0$ the spike at $t = 0$ is $B_0=\{0\}$ and if $X_0 = 1$, the spike at $t = 0$ is $B_0=\{1\}$.  If $x_0 \in (0,1)$, the spike at $t = 0$ is an interval of the form $[Y, 1]$ with probability $x_0$, and $[0,Y]$ with probability $1-x_0$. The probability density of $Y$ in the first case is $\frac{1-x_0}{x_0} (1-y)^{-2} \mathds{1}_{\{0 <y<x_0\}}$, and the probability density of $y$ in the range of the second case is $\frac{x_0}{1-x_0} y^{-2}\mathds{1}_{\{x_0 < y < 1\}}$. 

 \item {\bf Simulation of the other spikes $B_t$, $t>0$:}
 
Sample spikes $\left( t, M_t \right)$ following a Poisson point process with intensity $\left( dt \otimes \frac{dm}{m^2} \mathds{1}_{0<m<1} \right)$. Then when the current state is $0$, rescale time by a factor $\left|b(0)\right|$ and every spike $(t, M_t)$ yields an upward segment $[0, M_t]$. When the current state is $1$, rescale time by a factor $\left|b(1)\right|$ and the spike is a downward segment $[1-M_t, 1]$.
\end{enumerate}
\end{proposition}

\begin{proof}
If $t=0$, $B_0$ is the segment spanned by the Brownian motion $\beta$ starting at $x_0$ until it hits $\{0, 1\}$. Thanks to this fact and standard stopping time arguments, let us analyze $B_0$ and thus prove the second point. We see that $\{ B_0 = [0, Y] \} = \{ \tau_1(\beta) < \tau_0(\beta) \}$ while $\{ B_0 = [Y, 1] \} = \{ \tau_0(\beta) < \tau_1(\beta) \}$. We deal with the first case leaving the second to the reader. The event $\{ B_0 = [0, Y] \}$ occurs exactly when $\{ \tau_1(\beta) < \tau_0(\beta) \} = \{ \xb_0 = 0 \}$ which happens with probability $\P_{x_0}\left( \tau_1(\beta) < \tau_0(\beta) \right) = 1-x_0$. Moreover for $0 \leq x_0 \leq y < 1$:
$$ \P_{x_0}\left( y \leq Y, \ B_0 = [0,Y] \right) 
 = \P_{x_0}\left( \tau_y(\beta) < \tau_0(\beta) \right) = \frac{x_0}{y} \ .$$
As such, as long as $0 \leq x_0 \leq y < 1$:
$$ \P_{x_0}\left( Y \in dy \right) 
 = -\frac{d}{dy} \left( \frac{\P_{x_0}\left( y \leq Y, \ B_0 = [0,Y] \right)}{\P_{x_0}\left( B_0 = [0,Y] \right)} \right) dy
 = \frac{x_0}{1-x_0} \frac{dy}{y^2} \ .$$

We can now focus on the point (3) dealing with $B_t$ for $t>0$. Itô's Excursion Theory for Brownian motion states that \cite[Chapter XII]{RY} the trajectory of $\beta$ can be split into excursions indexed by the inverse of local time. Because of the definition of $\sigma$, we are dealing with a mixed local time and excursions around the points $0$ and $1$. Since $\xb_t = \beta_{\sigma_t}$ and the definition of $B_t$, we have for $0 < t \leq H$:
\begin{itemize}
\item Either $\xb_t = \xb_{t^-} = 0$: We are looking at an excursion around $0$ and $B_t = [0, \max_{[\sigma_{t^-}, \sigma_t]} \beta]$.
\item Either $\xb_t = \xb_{t^-} = 1$: We are looking at an excursion around $1$ and $B_t = [\min_{[\sigma_{t^-}, \sigma_t]} \beta, 1]$.
\item Otherwise, $\xb_t \neq \xb_{t^-}$ and we are looking at an excursion between $0$ and $1$. Here $B_t = [0,1]$ because of the intermediate value theorem.
\end{itemize}

Thanks to the Markov property, we only need to give an excursion point of view of the process for $x_0=0$ and up to hitting $1$, and for $x_0=1$ and up to hitting $0$. We shall focus on the first case and leave the other one to the reader. Let $H = L^0(\beta)^{\langle-1\rangle}$ be the inverse of local time and $\tau_1(\beta)$ the time at which $1$ is reached. On the segment $[0, \tau_1(\beta)]$, we have:
\begin{align}
\label{eq:time_change}
\forall \ t \in [0, T_1], \quad \sigma_t & = H_{2 \left|b(0)\right|  t} \ . 
\end{align}
Now recall that the Brownian path $\beta$ can be broken into a Poisson process of excursions away from zero $\left( e_t, \ t>0 \right)$. As such $(t, e_t)$ has intensity $\left( dt \otimes n \right)$, where $n$ is the Itô measure on excursions and the time scale is that of $H$ i.e. there is an excursion for every $t$ such that $H_{t}-H_{t^-}>0$. By changing to the time scale of $\sigma$ via Eq. \eqref{eq:time_change}, $(t, e_t)$ has intensity $\left( {2 \left|b(0)\right|} \, dt \otimes n \right)$ and there is an excursion every $t$ such that $\sigma_t-\sigma_{t^-}>0$.

Moreover the Itô measure restricted to positive excursions gives intensity $\half \frac{dm}{m^2} \mathds{1}_{0<m}$ to the decoration by maxima \cite[Chapter XII, Theorem 4.5]{RY}. Now, the process has to be killed at the first excursion of height $\geq 1$. Because of the thinning property of Poisson processes, one still has a Poisson process and the decoration by maxima has intensity $\half \frac{dm}{m^2} \mathds{1}_{0<m<1}$.

In the end, by looking only at the maxima $M_t = \sup_s e_t(s)$ of positive excursions $e_t$, $\left(t, M_t\right)$ is a Poisson process with intensity $\left|b(0)\right| \,  dt \otimes \frac{dm}{m^2} \mathds{1}_{0<m<1}$, in the time scale of $\sigma$.
\end{proof}

\section{The setup of Rabi oscillations} 

In this last section, motivated by another model in quantum continuous measurements, we consider an SDE where the drift $b$ depends on $\gamma$ and where the boundaries have different behaviors. 

\medskip

{\bf Description of the model:} Actually, Rabi oscillations form the basis of the definition of the second as a unit of time inside atomic clocks. More generally and as explained in \cite[Section 5.2]{L11}, Rabi oscillations play an important role in many engineering applications such  as NMR (Nuclear Magnetic Resonance), MRI (Magnetic Resonance Imaging) or quantum computing. 

The physical context is the following. We consider a qubit exactly as in the context of Eq. \eqref{eq:BBT_SDE}. This time however, there is no thermal bath and the system is subject to an external magnetic field with intensity $w \sqrt{\gamma} \in \R^*$. Even more importantly, the measured observable is not the Hamiltonian of the qubit inside the magnetic field and in fact does not commute with it. As such, whereas Eq. \eqref{eq:BBT_SDE} was about a competition between measurements and thermalization, the resulting SDE translates a competition between measurements and free evolution of the system inside the magnetic field.

As made precise in \cite[Section 2]{BB18}, after a few reductions of the state space, the measurement of Rabi oscillations is completely determined by an angle $\theta_t^\gamma \in \R$ whose dynamic is given by:
\begin{align}
\label{eq:def_rabi}
     d\theta_{t}^\gamma
 = & \left(-\frac{\gamma}{4}\sin\left(2\theta^\gamma_{t}\right)+w\sqrt{\gamma}\right) dt-\sqrt{\gamma}\sin\left(\theta^\gamma_{t}\right)dW_{t} \ .
\end{align}

\begin{figure}[ht]
\includegraphics[scale=0.25]{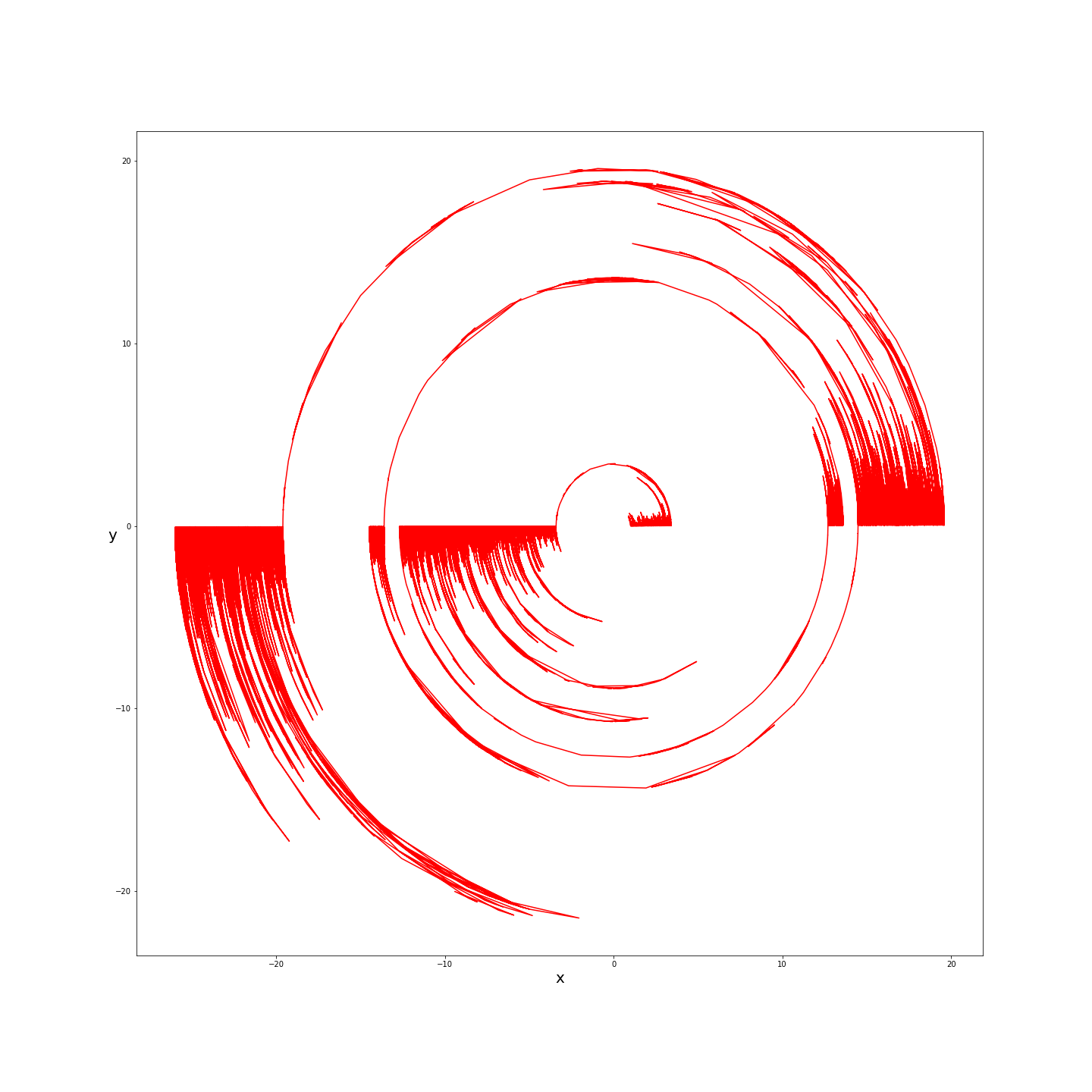}
\caption{Numerical simulation of the process $\left( \theta_t^\gamma; t \geq 0 \right)$ with $w=1.0$ and $\gamma=10^4$. There are $10^6$ time steps and we have plotted the parametric curve $t \mapsto \left( 5t\cos(\theta_t^\gamma), 5t\sin(\theta_t^\gamma) \right)$ for $t \in [0,5]$. The code is also available at the online repository \url{https://github.com/redachhaibi/quantumCollapse}}
\label{fig:rabiRadial}
\end{figure}

When relevant, we will write $\theta^\gamma=\theta^{\gamma, a}$ to specify the initial condition $\theta^\gamma_0=a$. Although we shall adopt the same approach as the one developed for Assumption \ref{assumption:main}, we will see this equation introduces some subtleties. Let us begin by a few reductions. Upon changing the space variable $\theta$ to $-\theta$, one can suppose that $w>0$. Also, since the diffusion's characteristics $\left( b, \sigma^2 \right)$ are $\pi$-periodic, we have the equality in law between processes:
\begin{align}
\label{eq:periodicity_rabi}
	\left( \theta_t^{\gamma, \theta_0} + \pi \ ; t \geq 0 \right)
	\eqlaw
	\left( \theta_t^{\gamma, \theta_0 + \pi} \ ; t \geq 0 \right)
	\ .
\end{align}
As such, without loss of generality, we can assume that $\theta_0 \in (0, \pi]$.

\medskip

{\bf Results: } First of all, the qualitative properties are very different upon fixing $\gamma>0$. This can be seen by analyzing hitting times. Recall that $\tau_a = \tau_a(X)$ denotes the hitting time of $a$ for a process $X$. We have:
\begin{lemma}
\label{lemma:qualitative_rabi}
The process $\theta^\gamma = \left( \theta_t^\gamma \ ; \ t \in \R_+ \right)$ passes through the consecutive intervals $\Z \pi + [0, \pi]$ increasingly, that is to say that each multiple of $\pi$ is reached only once, in an increasing order.
\end{lemma}
\begin{proof}
Let $k \in \Z$ be fixed.
\begin{itemize}
\item {\it{For any $a\in (k\pi, (k+1)\pi]$ the process $\theta^{\gamma,a}$ reaches $(k+1)\pi$ with probability $1$.}} Thanks to Eq. \eqref{eq:periodicity_rabi}, it suffices to prove it for $k=0$. Consider the process starting from $a \in (0, \pi]$ and stopped upon hitting $\pi$. Then the process $\theta^\gamma$ never reaches $0$ and leaves the domain from $\pi$. This is a simple consequence of the fact that  the scale function $h_\gamma$ defined by \eqref{eq:scaleRabi} satisfies $h_\gamma(0^+) = -\infty$, as we shall see, and that for any $0<x<a<y \leq \pi$, hitting probabilities are defined in term of the scale function via:
$$
	\P_a\left( \tau_x < \tau_y \right)
	= 
	\frac{h_\gamma(y) -  h_\gamma(a)}
	     {h_\gamma(y) -  h_\gamma(x)}
	\ .
$$
Indeed, taking $x \rightarrow 0^+$ in the above equation yields that for $0 < a < y \leq \pi$:
$$ \P_a\left( \tau_0 < \tau_y \right) = 0 \ .$$
The process $\theta^\gamma$ starting from $a$ reaches any point in $(a, \pi]$ before $0$, for every finite $\gamma$. In particular, the process reaches $\pi$ almost surely.

\item {\textit{$\theta^{\gamma, k\pi}$ leaves $k\pi$ from the right.}} Thanks to Eq. \eqref{eq:periodicity_rabi}, it suffices to prove it for $k=0$. And indeed, because of the previous argument, for any starting point $a>0$, the reaching time of $-\varepsilon<0$ is almost surely $\tau_{-\varepsilon}=\infty$. Since the SDE has smooth coefficients, the solution $\theta^{\gamma, a}$ is continuous in the initial condition $a$. Therefore $\theta^{\gamma, a} \rightarrow \theta^{\gamma, 0}$ as $a \to 0$ and $\tau_{-\varepsilon}=\infty$ starting from $0$ as well. This means that, starting from zero, the process never leaves the non-negative half-line $[0, \infty)$. 
\end{itemize}
\end{proof}

Now let us define the limiting processes for this instance. On the one hand, we define the half-winding number $\left( \bm{\vartheta}_t \ ; \ t \geq 0 \right)$ as $\pi$ times a Poisson process on $\R_+$ and intensity $2w^2$. If $\theta_0 \in [0, \pi)$, then the initial position $\bm{\vartheta}_0$ takes values in $\{0, \pi\}$ almost surely with:
$$ \P\left( \bm{\vartheta}_0 = 0 \right) = 1 - \P\left( \bm{\vartheta}_0 = \pi \right) = \sin^2\left( \frac{\theta_0}{2} \right) \ .$$

On the other hand, the spike process is the {\it set-valued} random path $\Theta: \R_+ \rightarrow \Pc\left( \R \right)$ obtained as follows:
\begin{itemize}
\item Sample a random initial segment as
$$
   \Theta_0 = \left\{
\begin{array}{ll}

[0,  Z] \textrm{ when $\bm{\vartheta}_0 = 0$ }  \ , \\

[Z,\pi] \textrm{ when $\bm{\vartheta}_0 = \pi$} \ .

\end{array}
   \right.
$$
In the first case, pick $Z$ according to
$$ 
   \P\left( Z \in dz \ | \ \bm{\vartheta}_0 = 0 \right)
 = \tan^2\left( \frac{\theta_0}{2} \right) \mathds{1}_{\{\theta_0 < z < \pi\}}
   \frac{2 \sin(z)}{\left( 1 - \cos(z) \right)^2} dz \ ,
$$
while in the second case:
$$ 
   \P\left( Z \in dz \ | \ \bm{\vartheta}_0 = \pi \right)
 = \cotan^2\left( \frac{\theta_0}{2} \right) \mathds{1}_{\{0 < z < \theta_0 \}}
   \frac{2 \sin(z)}{\left( 1 + \cos(z) \right)^2} dz \ .
$$
\item Sample $\left( t, M_t \right)$ following a Poisson point process on $\R_+ \times [0,\pi]$ with intensity
$$ \left( 2w^2 dt \otimes dm \frac{\sin\left( m \right)}{\left( 1 - \cos(m) \right)^2} \mathds{1}_{\{ 0 \leq m < \pi \}} \right) \ .$$
\item Finally
$$
   \Theta_t = \left\{
\begin{array}{ccc}
\ \bm{\vartheta}_t + [0, M_t]        & \textrm{ if } &  \bm{\vartheta}_t = \bm{\vartheta}_{t^-} \ ,\\
\ \bm{\vartheta}_{t^-} + [0, \pi]    & \textrm{ if } & \bm{\vartheta}_t \neq \bm{\vartheta}_{t^-} \ .
\end{array}
   \right.
$$
\end{itemize}

The analogue of Theorem \ref{thm:limit_thm_1d} is:
\begin{thm}[Rabi oscillation]
\label{thm:rabi}
For any $\gamma>0$ consider the process $\left( \theta_t^\gamma, \ t \geq 0 \right)$, unique strong solution of Equation \eqref{eq:base} starting from a given initial condition $\theta_0$ independent of $\gamma$. All these processes can be coupled for different $\gamma>0$ with $(\bm{\vartheta}, \Theta)$ so that the following two convergences hold in probability \footnote{Notice that unlike Theorem \ref{thm:limit_thm_1d}, we do not have almost sure convergence.}.

On the one hand, for every continuous and compactly supported function $f$ on $\R_+ \times \R$:
\begin{align}
\label{eq:limit_process_xb_rabi}
   \lim_{\gamma \rightarrow \infty}
   \int_0^\infty f( t, \theta_t^\gamma) dt
   =
   \int_0^\infty f( t, \bm{\vartheta}_t ) dt
   \ .
\end{align}

On the other hand, we have for all $H>0$ {\it the Hausdorff convergence of graphs}:
\begin{align}
\label{eq:limit_process_X_rabi}
   \left( \theta_t^\gamma \ ; \ 0 \leq t \leq H \right)
   \stackrel{\gamma \rightarrow \infty}{\longrightarrow}
   \left( \Theta_t \ ; \ 0 \leq t \leq H \right)
   \ .
\end{align}
   
\end{thm}

Note that the intensity of the half-winding number $\bm{\vartheta}$ is $2w^{2}$, while the value of the drift of SDE \eqref{eq:def_rabi} in $0$ or $\pi$ is $w\sqrt{\gamma}$, which diverges for large $\gamma$. This shows that not only the details of the proof are different compared to Theorem \ref{thm:limit_thm_1d}, but that the context of Rabi measurements is also very different at the physical level. 

If the term $w\sqrt{\gamma}$ in Eq. \eqref{eq:def_rabi} is substituted by $w\gamma^{a}$ with $a<\frac{1}{2}$, then the proof adapts and the limiting process $\bm{\vartheta}$ is constant. This is also seen in taking $w$ to zero in the previous Theorem. Such a freezing is known as the quantum Zeno effect \cite{MS77}. If on the other hand $a>\half$, there is no limit. We have thus studied the only relevant case.

\bigskip

The general structure of the proof for Theorem \ref{thm:rabi} goes as follows. In Subsection \ref{subsection:rabi_1segment}, we study the process $\theta^\gamma$ on $[0, \pi]$ until it hits $\pi$. Finally, we conclude in Subsection \ref{subsection:rabi_gluing} with the delicate issue of slicing trajectories into portions on $[0, \pi] + \pi \Z$, taking the limits and then gluing everything together.

\subsection{Studying the process on \texorpdfstring{$[0, \pi]$}{[0, pi]}}
\label{subsection:rabi_1segment}

Consider the process 
$$ \left( \theta_t^\gamma \ ; \ 0 \leq t \leq \tau_\pi(\theta^\gamma) \right) \ ,$$
unique strong solution of Equation \eqref{eq:def_rabi} starting from $\theta_0 \in (0, \pi)$. We start again by coupling all the solutions for different $\gamma>0$ thanks to a single Brownian motion $\beta$. This is done by considering the scale function $h_\gamma$ defined by \eqref{eq:scaleRabi}, and writing thanks to the Dambis-Dubins-Schwarz Theorem:
$$ h_\gamma(\theta_0) + \beta_{T_t} = h_\gamma \left( \theta_t^\gamma \right) \ ,$$
where $T_t:=T_t^{\gamma}$ is a suitable time-change. Here, we shall prove:
\begin{thm}[Rabi oscillation]
\label{thm:rabi2}
The following two statements hold almost surely.

On the one hand, for every continuous and compactly supported function $f$ on $\R_+ \times \R$:
\begin{align}
\label{eq:limit_process_xb_rabi2}
   \lim_{\gamma \rightarrow \infty}
   \int_0^{\tau_\pi(\theta^\gamma)} f( t, \theta_t^\gamma) dt
   & =
   \int_0^{\tau_\pi(\bm{\vartheta})} f( t, \bm{\vartheta}_t ) dt
   \ .
\end{align}

On the other hand, we have {\it the Hausdorff convergence of graphs}:
\begin{align}
\label{eq:limit_process_X_rabi2}
   \left( \theta_t^\gamma \ ; \ 0 \leq t \leq \tau_\pi(\theta^\gamma) \right)
   \stackrel{\gamma \rightarrow \infty}{\longrightarrow}
   \left( \Theta_t \ ; \ 0 \leq t \leq \tau_\pi(\bm{\vartheta}) \right)
   \ .
\end{align}

\end{thm}

First, let us study the scale function and describe some qualitative properties of the solutions.

\subsubsection{Rescaling space and time}

\begin{itemize}

\item{\bf Scale function:} For this problem, the scale function $h_\gamma: (0, \pi] \rightarrow \R$ satisfies
$$ \half \gamma \sin^2 (\theta) h_\gamma''(\theta)
 + \left( -\frac{\gamma}{4} \sin(2\theta) + \sqrt{\gamma} w \right) h_\gamma'(\theta) = 0
\ ,$$
or equivalently after using the duplication formula 
$$\sin(2\theta) = 2 \cos(\theta) \sin(\theta) \ ,$$
we have:
$$ \frac{h_\gamma''}{h_\gamma' }\, (\theta)
 = \cotan(\theta) - \frac{2w}{\sqrt\gamma} \frac{1}{\sin^2(\theta)}
\ .$$
Recall that a scale function is unique up to affine transformations. As such, we can choose $h_\gamma(\pi)=1$ and $h_\gamma^\prime (\pi)=0$. Upon solving that ODE with conveniently chosen constants, we have:
\begin{align}
\label{eq:scaleRabi}
    h_\gamma(\theta)
= & \ 1
    - \half \int_{\theta}^\pi d\varphi \sin(\varphi) e^{\frac{2w}{\sqrt\gamma} \cotan(\varphi) } \ .
\end{align}

\medskip

\item {\bf Time change:} As before, we write:
$$ \forall \,  t \in \left[0, \tau_\pi \left( \theta^\gamma \right)\right], \quad
   \theta_t^\gamma
 =  \left( h_\gamma^{\langle -1 \rangle} \circ h_\gamma \right) \left( \theta_t^\gamma \right)
 = h_\gamma^{\langle -1 \rangle}\left( h_\gamma(\theta_0) + \beta_{T_t} \right) \ ,
$$
where $\beta$ is a Brownian motion starting from zero. The time change is:
$$ T_t:=T_t^\gamma = \gamma \int_0^t \left( h_\gamma'( \theta^\gamma_s ) \right)^2 
                                \sin^2(\theta^\gamma_s) \ ds \ .
$$
Notice that this time, we stop the process $\theta^\gamma$ upon hitting $\pi$, which amounts to stopping $h_\gamma(\theta_0) + \beta$ upon hitting $h_\gamma(\pi) = 1$. The previous structure of proof still applies modulo technical details. To see that, we need to check that the inverse time-change behaves as expected. We write $\tau = \tau_\pi\left( h_\gamma(\theta_0) + \beta \right)$ and consider the inverse time-change up to $\tau$ only:
$$ \left( T_\ell^{\langle -1 \rangle}
   \ ; \ \ell < \tau
   \right) \ .$$
In this range, we have:

\begin{align}
\label{eq:time_change_rabi}
   dT_\ell^{\langle -1 \rangle}
 = & \ 
   \frac{d\ell}
        {\gamma \left[ \left( h_\gamma' \circ h_\gamma^{\langle -1 \rangle} \right)( h_\gamma(\theta_0) + \beta_\ell ) \right]^2 
        \left( \sin^2 \circ \, h_\gamma^{\langle -1 \rangle} \right) ( h_\gamma(\theta_0) + \beta_\ell )  }\\
 \nonumber
 =: & \ \varphi_\gamma( h_\gamma(\theta_0) + \beta_\ell ) d \ell \  .
\end{align}

\end{itemize}

\subsubsection{Asymptotics for the changes of scale}

The analogue of Lemma \ref{lemma:asymptotics} in the case of Rabi oscillations is the following. Nevertheless, the computations are markedly different.

\begin{lemma}
\label{lemma:asymptotics_rabi}
On the one hand, the map $h_\gamma: (0, \pi] \rightarrow \R$ defined by \eqref{eq:scaleRabi} has range $(-\infty, 1]$. Its inverse $h_\gamma^{\langle -1 \rangle}$ converges uniformly on $(-\infty, 1)$ to the function $h_\infty^{\langle -1 \rangle}$ defined by:
\begin{align}
\label{eq:cv_h_rabi}
 h_\infty^{\langle -1 \rangle}(y) & 
:= \lim_{\gamma \rightarrow \infty} h_\gamma^{\langle -1 \rangle}(y)
 = \left\{
   \begin{array}{ccc}
   0 & \textrm{ if }  & y \leq 0       \\
   2 \arcsin \sqrt{y} & \textrm{ if } & 0 \leq y < 1\\
   \end{array}
   \right.
   \ .
\end{align}

On the other hand, we have the weak convergence for test functions supported on $(-\infty, 1)$:
\begin{align}
\label{eq:cv_varphi_rabi}
   \varphi_\gamma & 
   \stackrel{\gamma \rightarrow \infty}{ \longrightarrow }
   \frac{1}{2 w^2} \delta_0  \ .
\end{align}

\end{lemma}
\begin{proof}
From the expression \eqref{eq:scaleRabi} of the scale function
$$ \lim_{\theta \rightarrow 0} h_\gamma(\theta) = -\infty
   \ \ \textrm{and} \ \ 
   h_\gamma(\pi) = 1
$$
for every fixed $\gamma$. As such, $h_\gamma: (0, \pi] \rightarrow (-\infty, 1]$ is an increasing diffeomorphism. Because of the limit
$$
\lim_{\gamma \rightarrow \infty}
h_\gamma(\theta) = 
  \half - \half \cos(\theta)
  = \sin^2 \left( \frac{\theta}{2} \right)
  =: h_\infty(\theta) \ ,
$$
which is uniform on compact subsets of $(0, \pi]$, we readily have the convergence of the inverse to:
$$ \forall y \in [0, 1], \ h_\infty^{\langle -1 \rangle}(y) = 2 \arcsin \sqrt{y} .$$
The limit for $y \leq 0$ follows and uniform convergence is obtained via Dini's Theorem. We have proved Eq. \eqref{eq:cv_h_rabi}.

Now let us move to Eq. \eqref{eq:cv_varphi_rabi} by considering a smooth function $f$ with compact support in $(-\infty,1)$. More precisely, we consider $\varepsilon>0$ and $\gamma$ large enough so that $h_\gamma([\pi-\varepsilon, \pi]) \notin \textrm{supp} f$. In this case:
\begin{align*}
     \int_{\R}  \varphi_\gamma(a) f(a)\, da
 = & 2 \int_{0}^\pi \ h_\gamma'(\theta)
                           \frac{(f \circ h_\gamma) (\theta) }
                                {\gamma h_\gamma'(\theta)^2 \sin^2(\theta) }\, d\theta \\
 = & 2 \int_{0}^\pi  \frac{(f \circ h_\gamma) (\theta) }
                                {\gamma e^{\frac{2 w}{\sqrt\gamma} \cotan(\theta)} \sin^3(\theta) } \, d\theta \\
 = & 2 \int_{0}^{\pi-\varepsilon}
     \frac{( f \circ h_\gamma) (\theta) }
                                {\gamma e^{\frac{2 w}{\sqrt\gamma} \cotan(\theta)} \sin^3(\theta) } \, d\theta \\
 = & o_\gamma(1) +
     2 \int_{0}^{\frac{\pi}{2}}
     \frac{( f \circ h_\gamma) (\theta) }
                                {\gamma e^{\frac{2 w}{\sqrt\gamma} \cotan(\theta)} \sin^3(\theta) } \, d\theta \ .                                
\end{align*}
Here $o_\gamma(1)$ is a quantity which goes to zero as $\gamma \rightarrow \infty$. Again via the change of variable $u=e^{-\frac{2 w}{\sqrt\gamma} \cotan(\theta)}$, which is equivalent to $\theta=q_\gamma(u)=\arctan(-\frac{2w}{\sqrt \gamma\ln(u)})$, we have:
\begin{align*}
    \int_{\R}  \varphi_\gamma(a) f(a)\, da
& = o_\gamma(1) +
     \int_{0}^{\frac{\pi}{2}}
     \cfrac{d}{d\theta} \left( e^{-\frac{2 w}{\sqrt\gamma} \cotan(\theta)} \right)
              \frac{(f \circ h_\gamma )(\theta) }
                   {w \sqrt{\gamma} \sin(\theta) }\, d \theta\\
& = o_\gamma(1)
    + \int_{0}^{1}  \ \frac{(f \circ h_\gamma \circ q_\gamma) (u) }
                             {w\sqrt{\gamma} (\sin \circ  q_\gamma )(u) }\, du \ .
\end{align*}
Now let us look at $(h_\gamma \circ q_\gamma) (u)$ for $u\in (0,1)$. Since $q_\gamma(u) \stackrel{\gamma \rightarrow \infty}{\longrightarrow} 0$, we have:
$$ 1 = o_\gamma(1) + \half \int_{q_\gamma(u)}^\pi \ \sin(\varphi) \, d\varphi \ .$$
As such, we have
\begin{align*}
 (h_\gamma \circ q_\gamma) (u)
= & 1 - \half \int_{q_\gamma(u)}^\pi \sin(\varphi) e^{\frac{2w}{\sqrt\gamma} \cotan(\varphi) }\, d\varphi \\
= & - \half \int_{q_\gamma(u)}^\pi \sin(\varphi) \left( e^{\frac{2w}{\sqrt\gamma} \cotan(\varphi) } - 1 \right) \, d\varphi  + o_\gamma(1)\\
= & - \half \int_{q_\gamma(u)}^{\frac\pi2}\sin(\varphi) \left( e^{\frac{2w}{\sqrt\gamma} \cotan(\varphi) } - 1 \right)\, d\varphi  + o_\gamma(1) \\
= & - \half \int_{0}^{\frac\pi2} \sin(\varphi) \mathds{1}_{\{ q_\gamma(u) \leq \varphi \}}
      \left( e^{\frac{2w}{\sqrt\gamma} \cotan(\varphi) } - 1 \right) \, d\varphi + o_\gamma(1) \ .
\end{align*}
Notice that the integrand converges Lebesgue almost everywhere to zero and that it is dominated by:
\begin{align*}
   \sup_{\varphi \in [q_\gamma(u), \frac\pi2]}
        \left| \sin(\varphi) \mathds{1}_{\{ q_\gamma(u) \leq \varphi \}}
               \left( e^{\frac{2w}{\sqrt\gamma} \cotan(\varphi) } - 1 \right) \right|
   \leq & \ \left| e^{ \frac{2w}{\sqrt\gamma} (\cotan \circ q_\gamma) (u) } - 1 \right| \\
   \leq & \ \exp\left( \frac{2w}{\sqrt\gamma} \cdot \frac{-\sqrt{\gamma} \ln u}{2w} \right) + 1 \\
   =    & \ u^{-1} + 1 \ .
\end{align*}
By Lebesgue's Dominated Convergence Theorem, we deduce that:
$$ \forall u \in (0,1), \quad \lim_{\gamma \rightarrow \infty} (h_\gamma \circ q_\gamma) (u) = 0 \ .$$
Moreover:
$$ \sqrt{\gamma} (\sin \circ q_\gamma) (u) 
 = \sqrt{\gamma} (\sin \circ \arctan) \left( -\frac{2w}{\sqrt \gamma\ln(u)} \right)
 \stackrel{\gamma \rightarrow \infty}{\sim} -\frac{2w}{\ln(u)}
   \ .$$
As a consequence, by combining the two previous limits and dominated convergence, we have:
\begin{align*}
\int_{\R} \varphi_\gamma(a) f(a) da
& = o_\gamma(1)
    + \int_{0}^{1}  \ \frac{f \circ h_\gamma \circ q_\gamma(u) }
                             {w\sqrt{\gamma}\sin \circ q_\gamma(u) } \, du\\
& \stackrel{\gamma \rightarrow \infty}{\longrightarrow} f(0) \int_0^1 \frac{-\ln u}{2w^2}\, du\\
& = \frac{f(0)}{2w^2} \ .
\end{align*}
\end{proof}

\subsubsection{The limiting process \texorpdfstring{$\bm{\vartheta}$}{xb} and proof of Eq. \texorpdfstring{\eqref{eq:limit_process_xb_rabi2}}{}}

We just have proven all the ingredients in order to obtain the convergence to a jump Markov process upon smoothing. This is done more or less along the same lines as Subsection \ref{subsection:limiting_xb}. Here are the details.

First, recall that 
$$ \forall \,  t \in \left[0, \tau_\pi \left( \theta^\gamma \right)\right], \quad \theta_t^\gamma
 = (h_\gamma^{\langle -1 \rangle} \circ h_\gamma) ( \theta_t^\gamma )
 = h_\gamma^{\langle -1 \rangle}( h_\gamma(\theta_0) + \beta_{T_t} ) \ ,
$$
with $\beta_0 = 0$. As before, we couple the different processes $\left( \theta^\gamma \ ; \ \gamma > 0 \right)$ by using the same Brownian motion $\beta$. Also, recall that stopping $\theta^\gamma$ at $\pi$ amounts to stopping $h_\gamma(\theta_0) + \beta$ at $1 = h_\gamma(\pi)$. For $\ell<\tau_1(\beta)$, thanks to the occupation time formula, the inverse time is:
\begin{align}
\label{eq:T_inv_rabi}
     T_\ell^{\langle -1 \rangle}
 = \int_0^\ell \varphi_\gamma( h_\gamma(\theta_0) + \beta_\ell ) \, d \ell 
 = \int_{(-\infty,1]} \varphi_\gamma(a) L_\ell^a( h_\gamma(\theta_0) + \beta) \, da \  .
\end{align}

In fact, $T_\ell^{\langle -1 \rangle}<\infty$ while $\ell<\tau_1(h_\gamma(\theta_0) + \beta)$. This is because $\varphi_\gamma$ remains bounded away from $1$. The natural convention is to consider that $T_\ell^{\langle -1 \rangle} = \infty$ for $\ell=\tau_1(h_\gamma(\theta_0) + \beta)$, since $T$ will become flat by stopping $\theta^\gamma$ after $\tau_\pi(\theta^\gamma)$. In any case, notice now that $\left( L_\ell^a( h_\gamma(\theta_0) + \beta) \ ; \ a \in \R \right)$ is strongly convergent as $\gamma \rightarrow \infty$ since local time is continuous and compactly supported in $a \in \R$. Adding to that $\varphi_\gamma$ converges weakly via Lemma \ref{lemma:asymptotics_rabi}, we have the almost sure convergence:
\begin{align*}
  & \left( T^{\langle -1 \rangle}_\ell \ ; \ 0 \leq \ell < \tau_1(h_\gamma(\theta_0) + \beta) \right) \\
\stackrel{\gamma \rightarrow \infty}{\longrightarrow}
  & \left( \sigma^{\langle -1 \rangle}_\ell := \frac{L^{0}_\ell(h_\infty(\theta_0) + \beta)}{2 w^2} \ ; \ 0 \leq \ell < \tau_1(h_\infty(\theta_0) + \beta) \right) \ .
\end{align*}

In passing, let us notice that $T_{\tau_\pi(\theta^\gamma)} = \tau_1(h_\gamma(\theta_0)+\beta)$ which leads to following, which we record for later use:
\begin{lemma}
\label{lemma:cv_time}
$$ \tau_\pi\left( \theta^\gamma \right)
   \stackrel{\gamma \rightarrow \infty}{\longrightarrow}
   \frac{1}{2w^2}
   L^{0}_{\tau_1(h_\infty(\theta_0) + \beta)}(h_\infty(\theta_0) + \beta) \ .
$$
\end{lemma}

Then, taking $f$ compactly supported and performing a change of variable, we have:
\begin{align*}
    \int_0^{\tau_\pi(\theta^\gamma)}f\left( t, \theta_t^\gamma \right) dt
= & \int_0^{\tau_\pi(\theta^\gamma)}f\left( t, h_\gamma^{\langle -1 \rangle}( h_\gamma(\theta_0) + \beta_{T_t} ) \right) dt\\
= & \int_0^{\tau_1(h_\gamma(\theta_0)+\beta)}f\left( T^{\langle -1 \rangle}_{\ell}, h_\gamma^{\langle -1 \rangle}( h_\gamma(\theta_0) + \beta_\ell ) \right)
                  d T^{\langle -1 \rangle}_{\ell} \\
\stackrel{\gamma \rightarrow \infty}{\longrightarrow}
    & \int_0^{\tau_1(h_\infty(\theta_0)+\beta)}f ( \sigma^{\langle -1 \rangle}_\ell, h_\infty^{\langle -1 \rangle}(h_\infty(\theta_0)  + \beta_\ell) ) \, d \sigma^{\langle -1 \rangle}_\ell \ .
\end{align*}
Thanks to standard martingale arguments, the Brownian motion $h_\infty(\theta_0) + \beta$ reaches $1$ before $0$ with probability $\sin^2\left( \frac{\theta_0}{2} \right)$. On the corresponding event $\{\tau_1 (h_\infty(\theta_0)+\beta) < \tau_0 (h_\infty(\theta_0)+\beta)\}$ the above integral is thus zero. By setting $\bm{\vartheta}_0=\pi$ on this event, we have indeed on the latter (since $\tau_\pi (\bm{\vartheta})=0$):
$$ \int_0^{\tau_1(h_\infty(\theta_0) + \beta)}f ( \sigma^{\langle -1 \rangle}_\ell, h_\infty^{\langle -1 \rangle}(h_\infty(\theta_0)+\beta_\ell) ) \, d \sigma^{\langle -1 \rangle}_\ell
 = 0 =  \int_0^{\tau_\pi(\bm{\vartheta})} f(t, \bm{\vartheta}_t) dt \ .$$

With probability $1-\sin^2\left( \frac{\theta_0}{2} \right)$, $h_\infty(\theta_0)+\beta$ reaches $0$ first. On the corresponding event $\{\tau_0 (h_\infty(\theta_0)+\beta) < \tau_1 (h_\infty(\theta_0)+\beta)\}$, we set $\bm{\vartheta}_0=0$. Thanks to the Markov property and the fact that $\sigma^{\langle -1 \rangle}$ is supported on $\{ \ell \geq 0 \ | \ h_\infty(\theta_0)+\beta_\ell = 0 \}$, we can suppose that $\beta$ starts at zero in the above limit which will allow lighter notations. By performing a change of variable again, the latter expression equals
$$ \int_0^{\tau_1(\beta)}f ( \sigma^{\langle -1 \rangle}_\ell, 0 ) \, d \sigma^{\langle -1 \rangle}_\ell
 = \int_0^{\frac{L^0_{\tau_1(\beta)}(\beta)}{2w^2}} f ( t, 0 ) dt \ .$$
Thanks to the first Ray-Knight Theorem \cite[Chapter XI, Theorem 2.2]{RY}, the accumulated local time $L_{\tau_1(\beta)}^0(\beta)$ is the square of a two dimensional Bessel process at time $1$, or equivalently $2{\mathcal E} (1)$ where $\mathcal E (1)$ is a standard exponential random variable. In the end, $\frac{L^0_{\tau_1(\beta)}(\beta)}{2w^2} = \frac{ {\mathcal E} (1) }{w^2}$. This is exactly the waiting time for the Poisson process $\bm{\vartheta}$ to jump from the state $0$ to the state $\pi$. Thus, with such a coupling, we have indeed:
$$ \int_0^{\tau_1(\beta)}f ( \sigma^{\langle -1 \rangle}_\ell, h_\infty^{\langle -1 \rangle}(\beta_\ell) )\,  d \sigma^{\langle -1 \rangle}_\ell
 = \int_0^{\tau_\pi(\bm{\vartheta})} f(t, \bm{\vartheta}_t) dt \ .$$
 This concludes the proof of the first part of Theorem \ref{thm:rabi2}.
 
\subsubsection{The limiting process \texorpdfstring{$\Theta$}{Theta} and proof of Eq. \texorpdfstring{\eqref{eq:limit_process_X_rabi2}}{}}
For $\gamma>0$, let $A^\gamma$ be the graph of $\theta^\gamma$ until it reaches $\pi$. Via time change we have:
\begin{align*}
A^\gamma
  & = \left\{ (t, \theta_t^\gamma) \ | \ 0 \leq t \leq \tau_\pi(\theta^\gamma) \right\} \\
  & = \left\{ \left( T^{\langle -1 \rangle}_\ell, h^{\langle -1 \rangle}_\gamma(h_\gamma(\theta_0) + \beta_\ell) \right) \ | \ 0 \leq \ell \leq \tau_1(h_\gamma(\theta_0) + \beta) \right\} \\
  & \stackrel{\gamma \rightarrow \infty}{\longrightarrow}
      \left\{ \left( \sigma^{\langle -1 \rangle}_\ell, h_\infty^{\langle -1 \rangle}(h_\infty(\theta_0) + \beta_\ell) \right) \ | \ 0 \leq \ell \leq \tau_1(h_\infty(\theta_0) + \beta) \right\} \ .
\end{align*}
the limit holding almost surely in the Hausdorff topology as before. We need to prove that
\begin{align*}
  & \left\{ \left( \sigma^{\langle -1 \rangle}_\ell, h_\infty^{\langle -1 \rangle}(h_\infty(\theta_0) + \beta_\ell) \right)
           \ | \ 0 \leq \ell \leq \tau_1(h_\infty(\theta_0) + \beta)
    \right\} \\
= & \bigsqcup_{0 \leq t \leq \frac{1}{2w^2} L^0_{\tau_1(h_\infty(\theta_0)+\beta)}(h_\infty(\theta_0)+\beta) } \left\{ t \right\} \times B_t
\end{align*}
gives $\Theta$ as in the statement of Theorem \ref{thm:rabi2}. Equivalently, we need to prove that the spikes from the theorem can be identified to
\begin{align*} 
  & B_t\\
= & \left\{ \ h_\infty^{\langle -1 \rangle}(h_\infty(\theta_0)+\beta_\ell) \ | \ \sigma_\ell^{\langle -1 \rangle} = t \right\} \\
= & \left\{ \ 2 \arcsin\left( \sqrt{h_\infty(\theta_0)+\beta_\ell} \right) \ \mathds{1}_{\{ 0 \leq h_\infty(\theta_0) + \beta_\ell \leq 1 \}}
                \ | \ \frac{L_\ell^0(h_\infty(\theta_0)+\beta)}{2 w^2} = t \right\} \ .
\end{align*}

Let us start with analyzing $B_t$ for $t>0$. In that endeavor, we set a few notations and simplifications. Since the local time $L_\ell^0(h_\infty(\theta_0)+\beta)$ needs to be non-negative for $t>0$, the process $h_\infty(\theta_0)+\beta$ has passed by zero. Without loss of generality, we can assume that $h_\infty(\theta_0)+\beta$ starts from zero by the Markov property. If $H = L^0(\beta)^{\langle -1 \rangle}$ is the inverse of local time, then the excursion of $\beta$ at (inverse local) time $t$ is:
$$ e_t = \left\{ \ e_t(\ell) = \beta_{\ell+H_t} \ | \ 0 \leq \ell \leq H_{t} - H_{t^-} \right\} \ .$$
Also, its maximum is denoted by $\widetilde{M}_t := \sup_\ell e_t(\ell)$. We see that $B_t$ is defined in terms of the excursion $e_{2 w^2 t}$. It is clear  that $B_t$ is a segment in the form $[0,b_t]$ that can be expressed in terms of $\widetilde{M}$:
\begin{itemize}
\item Either $\widetilde{M}_{2w^2 t} = 0$, so that $e_{2 w^2 t}$ is negative. In this case, $b_t=0$.
\item or $\widetilde{M}_{2w^2 t} > 1$, so that $t = \tau_1(\beta)$ and $b_t=\pi$ is the last spike. 
\item or $0 < \widetilde{M}_{2w^2 t} < 1$, so that:
\begin{align*} 
b_t =2 \arcsin\left( \sqrt{\widetilde{M}_{2w^2 t}} \right) \ .
\end{align*}
\end{itemize}

Now, let us describe $(t, b_t)$ as a Poisson point process.  Recall that, $(t, \widetilde{M}_t)$ is a Poisson point  process with intensity $dt \otimes \half \frac{dm}{m^2} \mathds{1}_{m>0}$ from the description of the Itô measure and the decoration of positive excursions by maxima \cite[Chapter XII, Theorem 4.5]{RY}. Therefore the intensity is obtained by computing for a positive function $f$:
\begin{align*} 
    \E\left[ \sum_{0 < s \leq t} f\left(s, b_s\right) \right]
= & \E\left[ \sum_{0 < s \leq t} f\left(s, 2 \arcsin\left( \sqrt{\widetilde{M}_{2w^2 s}} \right) \right) \mathds{1}_{\{ 0 < \widetilde{M}_{2w^2 s} < 1 \}} \right]\\
= & \E\left[ \sum_{0 < s \leq 2w^2 t} f\left( \frac{s}{2w^2}, 2 \arcsin\left( \sqrt{\widetilde{M}_s} \right) \right) \mathds{1}_{\{ 0 < \widetilde{M}_s < 1 \}} \right]\\
= & \int_{s=0}^{2w^2 t} \int_{m=0}^1 \cfrac{1}{2 m^2}\,  f\left( \cfrac{s}{2w^2}, 2 \arcsin \sqrt{m} \right) \; ds \, dm\\
= & \int_{s=0}^t \int_{m=0}^\pi  \frac{2 w^2 \sin (m)}{\left( 1 - \cos(m) \right)^2}\,  f(s, m) \; ds\, dm\ .
\end{align*}
We are done with the description of $B_t$, $t>0$. Only the initial random segment $B_0$ needs to be described, and it coincides with the segment spanned by $h_\infty^{\langle -1 \rangle}(h_\infty(\theta_0)+\beta)$ until $\beta$ hits $\{0,1\}$. Using the same random variable $Y$ as in Proposition \ref{proposition:equivalent_MC}, we have that:
$$
   B_0 = \left\{
\begin{array}{ll}

[0,  2 \arcsin \sqrt{Y}] & \textrm{ if } \tau_0(h_\infty(\theta_0)+\beta) < \tau_1(h_\infty(\theta_0)+\beta) \ , \\

[2 \arcsin \sqrt{Y}, \pi] & \textrm{ otherwise } \ . \\
\end{array}
   \right.
$$
In the first case, we set $\bm{\vartheta}_0=0$ and this occurs with probability 
$$\P\left[ \tau_0(h_\infty(\theta_0)+\beta) < \tau_1(h_\infty(\theta_0)+\beta) \right] = \sin^2\left( \frac{\theta_0}{2} \right) \ .$$
In the second case, we set $\bm{\vartheta}_0=\pi$. Setting $Z = 2 \arcsin \sqrt{Y}$, and computing its density for each case yields the result i.e. the description of the initial spike $\Theta_0$. This concludes the proof of the second part of Theorem \ref{thm:rabi2}, and thus the goal of this subsection.\\

The full proof of Theorem \ref{thm:rabi} requires taking care of the following technicality. We will need to glue together trajectories belonging to $k\pi + [0, \pi]$ for different $k \in \Z$.

\subsection{Slicing and gluing trajectories}
\label{subsection:rabi_gluing}

Thanks to the remark made in Eq. \eqref{eq:periodicity_rabi}, the study of the full process can be decomposed into intervals of length $\pi$. Since the scale function is only defined on $(0, \pi]$ while the process $\theta^\gamma$ does cross $\pi$, the change of space and time does not work beyond that point. Likewise, it does not work when $\theta^\gamma \in \pi \Z$. Therefore, we need to treat the different portions of the trajectory separately. To simplify notation we denote the hitting times $ \tau_a = \tau_a(\theta^\gamma)
 = \inf\{ t \geq 0 \ | \ \theta^\gamma_t \geq a \} $
without the explicit reference to $\theta^\gamma$.

\subsubsection{Coupling}
The coupling of the processes $\theta^\gamma$ for different $\gamma>0$ requires some care, as we shall consider the segments $k\pi + [0, \pi]$ separately for different $k \in \Z$. Furthermore, we also need an epsilon of room $\varepsilon>0$, in order to avoid $0^+ + \pi \Z$. We start from independent Brownian motions $\left( \beta^{k} \ ; k \in \Z \right)$. For each of those, we consider the time change $T^k$ whose inverse is defined exactly as in Eq. \eqref{eq:time_change_rabi} via:
$$
   d\left( T^k \right)^{\langle -1 \rangle}_\ell
 = \varphi_\gamma( h_\gamma(\varepsilon) + \beta^k_\ell ) d \ell \  .
$$
As such, we unambiguously define portions of trajectory of $\theta^\gamma$ via:
\begin{align}
\label{eq:coupling_rabi}
\theta^\gamma_{\tau_{k\pi+\varepsilon} + t} = & k \pi + h_\gamma^{\langle -1\rangle}\left( h_\gamma(\varepsilon) + \beta^{k}_{T_t^k} \right) \ .
\end{align}
for $k\in \Z$ and $0 \leq t \leq \tau_{(k+1)\pi}-\tau_{k\pi+\varepsilon}$. For the portions of the trajectory with $t \in \left[ \tau_{k\pi}, \tau_{k\pi+\varepsilon} \right]$, we simply solve the original SDE \eqref{eq:def_rabi} driven by a Brownian motion independent from the $\left( \beta^{k} \ ; k \in \Z \right)$. By doing so, we have thus defined a family of processes $\theta^\gamma$ indexed by all of $t \in \R_+$, where each portion living on $\left[ \tau_{k\pi+\varepsilon}, \tau_{(k+1)\pi} \right]$ is coupled for different $\gamma>0$ thanks to independent Dambis-Dubins-Schwarz Brownian motions.

Notice that if we had first solved the SDE \eqref{eq:def_rabi} and then invoked Eq. \eqref{eq:coupling_rabi} to define the Brownian motions $\beta^{k}$, the dependence in $\varepsilon>0$ would not be tractable. Indeed, the outcome of the Dambis-Dubins-Schwarz Theorem is very sensitive to choice of filtration given for different $\varepsilon>0$.

\subsubsection{The jump process \texorpdfstring{$\bm{\vartheta}$}{x}}
For the proof of Eq. \eqref{eq:limit_process_xb_rabi}, there is no loss of generality in assuming that $f$ is smooth. Moreover, since $f$ has compact support, there exist $T>0$ and $K \in \N$ such that
$$ \textrm{supp} f \subset [0,T] \times [0, (K+1) \pi] \ .$$
We slice the integral into portions corresponding to the different intervals $k\pi + [0, \pi]$:
\begin{align*}
  & \int_0^\infty f\left( t, \theta_t^\gamma \right) dt \\
= & \sum_{k=0}^K
    \int_{\tau_{k\pi}}^{\tau_{(k+1)\pi}}
    f\left( t, \theta^\gamma_t \right) dt\\
= & \sum_{k=0}^K
    \int_{\tau_{k\pi}}^{\tau_{k\pi+\varepsilon}}
    f\left( t, \theta_t^\gamma \right) dt
    +
    \int_{\tau_{k\pi+\varepsilon}}^{\tau_{(k+1)\pi}}
    f\left( t, \theta_t^\gamma \right) dt
    \\
= & \sum_{k=0}^K
    \int_{\tau_{k\pi}}^{\tau_{k\pi+\varepsilon}}
    f\left( t, \theta_t^\gamma\right) dt \\
  & \quad 
    +
    \int_{0}^{\tau_{(k+1)\pi}-\tau_{k\pi+\varepsilon}}
    f\left( \tau_{k\pi+\varepsilon}+t, k \pi + h_\gamma^{\langle -1 \rangle}( h_\gamma(\varepsilon) + \beta_{T_t^k}^k ) \right) dt
    \\
= & \Oc\left( \sum_{k=0}^{K} (\tau_{k\pi+\varepsilon}-\tau_{k\pi}) \|f\|_\infty \right) \\
  & \quad +
	\sum_{k=0}^K
    \int_{0}^{\tau_{(k+1)\pi}-\tau_{k\pi+\varepsilon}}
    f\left( \tau_{k\pi+\varepsilon}+t, k \pi + h_\gamma^{\langle -1 \rangle}( h_\gamma(\varepsilon) + \beta_{T^k_t}^k ) \right) dt
    \ .
\end{align*}
Now, define the sequence of times
$$ \sigma_k^{\gamma, \varepsilon} := \sum_{j < k} \tau_{(j+1)\pi}-\tau_{j\pi+\varepsilon} \ .$$
By adding to $\sigma_k^{\gamma, \varepsilon}$ the duration of the intervals $\left( \left[ \tau_{j\pi}, \tau_{j\pi+\varepsilon} \right] \ ; \ j \leq k \right)$, we obtain $\tau_{k\pi+\varepsilon}$ as a telescoping sum and as such:
$$ \forall k \in \{0, \dots, K\},
 \ \left| \sigma_k^{\gamma, \varepsilon} - \tau_{k\pi+\varepsilon} \right|
   \leq \sum_{k=0}^{K} (\tau_{k\pi+\varepsilon}-\tau_{k\pi}) \ .$$
Thanks to the Mean Value Theorem, we obtain:
\begin{align*}
  & \int_0^\infty f\left( t, \theta_t^\gamma \right) dt \\
= & \Oc\left( \sum_{k=0}^{K} (\tau_{k\pi+\varepsilon}-\tau_{k\pi})( \|f\|_\infty + T \|\partial_t f\|_\infty ) \right) \\
  & \quad +
	\sum_{k=0}^K
    \int_{0}^{\tau_{(k+1)\pi}-\tau_{k\pi+\varepsilon}}
    f\left( \sigma_k^{\gamma, \varepsilon} +t, k \pi + h_\gamma^{\langle -1 \rangle}( h_\gamma(\varepsilon) + \beta_{T^k_t}^k ) \right) dt
    \ .
\end{align*}

Thanks to the estimate of the next subsection, in Lemma \ref{lemma:estimate}, the first term converges to zero in probability as $\varepsilon \rightarrow 0$, {\it uniformly} in $\gamma$. This is written as a $o_{\varepsilon, \P}(1)$. Applying successively the convergence result for $[0, \pi]$ in Lemma \ref{lemma:cv_time}, we have couplings such that
$$ \tau_{(k+1)\pi}-\tau_{k\pi+\varepsilon} \stackrel{\gamma \rightarrow \infty}{\longrightarrow}
   \frac{1}{2w^2}
   L^{0}_{\tau_1(h_\infty(\varepsilon) + \beta^k)}(h_\infty(\varepsilon) + \beta^k) 
$$
for $k\geq 1$. And
$$ \tau_{\pi} \stackrel{\gamma \rightarrow \infty}{\longrightarrow}
   \frac{1}{2w^2}
   L^{0}_{\tau_1(h_\infty(\theta_0) + \beta^0)}(h_\infty(\theta_0) + \beta^0) \ .
$$
As such, appears the quantity:
$$ \sigma^{\infty, \varepsilon}_k
   := \frac{1}{2w^2}
   L^{0}_{\tau_1(h_\infty(\theta_0) + \beta^0)}(h_\infty(\theta_0) + \beta^0)
   + \sum_{j=1}^k \frac{1}{2w^2}
   L^{0}_{\tau_1(h_\infty(\varepsilon) + \beta^j)}(h_\infty(\varepsilon) + \beta^j) \ .
$$
Continuing the computation, we have:
\begin{align}
\label{eq:rabi_final}
  & \ \int_0^\infty f\left( t, \theta_t^\gamma \right) dt \\
= & \ o_{\varepsilon, \P}(1)
	+
	o_{\gamma}(1)
	+
	\sum_{k=0}^K
    \int_{0}^{\sigma^{\infty, \varepsilon}_{k+1}-\sigma^{\infty, \varepsilon}_k}
    f\left( \sigma^{\infty, \varepsilon}_k +t, k \pi \right) dt
\nonumber
    \ .
\end{align}

Now recall that $\lim_{\varepsilon \rightarrow 0} h_\infty(\varepsilon) = 0$. As such we see that $\left( \sigma^{\infty, \varepsilon}_k \ ; \ k \in \N \right)$ converges to an increasing sequence $\left( \sigma^{\infty, 0}_k \ ; \ k \in \N \right)$. In turn, we can finally construct the Poisson process $\frac{1}{\pi} \bm{\vartheta}$ by taking these as jump times. We conclude by letting $\gamma \rightarrow \infty$, which is an almost sure limit, then $\varepsilon \rightarrow 0$, which is a limit in probability. Eq. \eqref{eq:rabi_final} becomes:
$$
   \P-\lim_{\gamma \rightarrow \infty} \int_0^\infty f\left( t, \theta_t^\gamma \right) dt 
 = \int_{0}^{\infty}
   f\left( t, \bm{\vartheta}_t \right) dt \ .
$$
.

\subsubsection{The spike process \texorpdfstring{$\Theta$}{Theta}}

Most of the work was already done. For $\gamma>0$, consider the graph of $\theta^\gamma$ up to a time horizon $H>0$:
$$
A^\gamma
= \left\{ (t, \theta_t^\gamma) \ | \ 0 \leq t \leq H \right\}
= \bigcup_{k=0}^\infty A^{\gamma}_k \ ,
$$ 
where
$$
   A_k^\gamma : = \left\{ (t, \theta_t^\gamma) \ | \ \tau_{k\pi} \wedge H \leq t \leq \tau_{(k+1) \pi} \wedge H \right\} \ .
$$
Our goal is to show the convergence of $A^\gamma$ to the spike process $\Theta$ restricted to $[0,H]$, which we write $\Theta_{|[0,H]}$. We require a convergence in probability for the Hausdorff distance ${\rm d}_{\H}$ and as such, we need to prove:
\begin{align*}
   \forall \delta>0, \ 
   \lim_{\gamma \rightarrow \infty} \P\left[ {\rm d}_{\H}\left( A^\gamma, \Theta_{|[0,H]} \right) \geq \delta \right]
   = 0 \ .
\end{align*}
By virtue of the treatment in Subsection \ref{subsection:rabi_1segment}, each portion $A_k^\gamma$ converges almost surely when restricted to $[\tau_{k\pi+\varepsilon}, \tau_{(k+1)\pi}]$. Moreover $\tau_{k\pi+\varepsilon}$ is close to $\tau_{k\pi}$ in probability and $\tau_{k\pi} = \tau_{k\pi}(\theta^\gamma)$ converges to $\sigma^{\infty, 0}_k$. Therefore, for fixed $k \in \N$, it is clear that each portion $A_k^\gamma$ converges to 
$$ \Theta^k = \left\{ (t, \Theta_t) \ | \ \sigma^{\infty, 0}_k \wedge H \leq t \leq \sigma^{\infty, 0}_{k+1} \wedge H\right\} \ ,$$
which is a portion of the spike process. These limits are in probability for the Hausdorff distance:
$$ \forall \delta>0, \ \P\left[ {\rm d}_{\H}(A^\gamma_k, \Theta^k) \geq \delta \right] \stackrel{\gamma \rightarrow \infty}{\longrightarrow} 0 \ .$$
By a union bound, we have for every fixed $K \in \N$:
$$ \forall \delta>0, \ \P\left[ \exists \ 0 \leq k \leq K, \ {\rm d}_{\H}\left( A^\gamma_k, \Theta^k \right) \geq \delta \right] \stackrel{\gamma \rightarrow \infty}{\longrightarrow} 0 \ .$$
Taking the complement:
$$ \forall \delta>0, \ \P\left[ \forall \ 0 \leq k \leq K, \ {\rm d}_{\H}\left( A^\gamma_k, \Theta^k \right) \leq \delta \right] \stackrel{\gamma \rightarrow \infty}{\longrightarrow} 1 \ .$$

By definition of the Hausdorff distance in Eq. \eqref{def:hausdorff}:
\begin{align*}
  & \left\{ \forall \ 0 \leq k \leq K, \ {\rm d}_{\H}\left( A^\gamma_k, \Theta^k \right) \leq \delta \right\} \\
= & \left\{ \forall \ 0 \leq k \leq K, \ A^\gamma_k \subset \Theta^k + \delta \B \textrm{ and vice-versa} \right\}\\
\subset
  & \left\{ \bigcup_{k=0}^K A^\gamma_k \subset \bigcup_{k=0}^K \Theta^k + \delta \B \textrm{ and vice-versa} \right\}\\
=  
  & \left\{ {\rm d}_{\H}\left( \bigcup_{k=0}^K A^\gamma_k, \bigcup_{k=0}^K \Theta^k \right) \leq \delta \right\} \ .
\end{align*}
This implies that:
\begin{align}
\label{eq:rabi_spike_cv1}
\forall \delta>0, \ \P\left[ {\rm d}_{\H}\left( \bigcup_{k=0}^K A^\gamma_k, \bigcup_{k=0}^K \Theta^k \right) \leq \delta \right] \stackrel{\gamma \rightarrow \infty}{\longrightarrow} 1 \ .
\end{align}

We conclude by remarking that taking $K \in \N$ large enough exhausts the segment $[0, H]$. This idea is formalized by the following inequalities:
\begin{align*}
      & \ \P\left[ {\rm d}_{\H}\left( A^\gamma, \Theta_{|[0,H]} \right) \geq \delta \right] \\
\leq  & \ \P\left[ \tau_{K\pi}(\theta^\gamma) \wedge \sigma_K^{\infty, 0} \geq H, {\rm d}_{\H}\left( A^\gamma, \Theta_{|[0,H]} \right) \geq \delta \right] \\
      &  \quad \quad \quad + \P\left[ \tau_{K\pi}(\theta^\gamma) \wedge \sigma_K^{\infty, 0} \leq H \right] \\
=     & \ \P\left[ \tau_{K\pi}(\theta^\gamma) \wedge \sigma_K^{\infty, 0} \geq H, {\rm d}_{\H}\left( \bigcup_{k=0}^K A^\gamma_k, \bigcup_{k=0}^K \Theta^k \right) \geq \delta \right] \\
      & \quad \quad \quad + \P\left[ \tau_{K\pi}(\theta^\gamma) \wedge \sigma_K^{\infty, 0} \leq H \right] \\
\leq  & \ \P\left[ {\rm d}_{\H}\left( \bigcup_{k=0}^K A^\gamma_k, \bigcup_{k=0}^K \Theta^k \right) \geq \delta \right] 
        + \P\left[ \tau_{K\pi}(\theta^\gamma) \wedge \sigma_K^{\infty, 0} \leq H \right] \ .
\end{align*}
Thanks to the convergence in Eq. \eqref{eq:rabi_spike_cv1} and the convergence in probability of $\tau_{K\pi}(\theta^\gamma)$ to $\sigma_K^{\infty, 0}$, we have for all $K \in \N$:
\begin{align}
\label{eq:rabi_spike_cv2}
   \limsup_{\gamma \rightarrow \infty} \P\left[ {\rm d}_{\H}\left( A^\gamma, \Theta_{|[0,H]} \right) \geq \delta \right]
   \leq 
   \P\left[ \sigma_K^{\infty, 0} \leq H+1 \right] \ .
\end{align}
Now recall that $\sigma_K^{\infty, 0}$ is the sum of $K$ exponential random variables, as the $K$-th jumping time of a standard Poisson process. Thus it converges to infinity as $K \rightarrow \infty$. Taking $K \rightarrow \infty$ in Eq. \eqref{eq:rabi_spike_cv2} concludes the proof. 

\subsection{An estimate}

We required in the previous Subsection \ref{subsection:rabi_gluing}:
\begin{lemma}
\label{lemma:estimate}
Let $\varepsilon>0$. If $\tau_{[0, \varepsilon]}$ denotes the first exit time from the interval $[0, \varepsilon]$, then:
\begin{align}
\label{eq:slicing3}
\sup_{\gamma  \geq 1}
\E_0\left( \tau_{[0, \varepsilon]} \right)
\underset{\varepsilon \rightarrow 0}{\longrightarrow} 0 \ .
\end{align}
\end{lemma}
\begin{proof}
Let $\Lc$ be the infinitesimal generator of $\theta^\gamma$. If we set for $0 \leq x < \pi$:
$$ f(x) := \E_x\left( \tau_{[0, \varepsilon]} \right) \ ,$$
then $f$ solves the Poisson equation:
$$
\left\{
\begin{array}{ccc}
\Lc f + 1 & = & 0 \ ,\\
f(\varepsilon)   & = & 0 \ .
\end{array}
\right.
$$
In explicit terms, this gives:
$$ 
  \frac{\gamma}{2} \sin^2(\theta) f''(\theta)
+ \left( \sqrt{\gamma} w - \frac{\gamma}{4} \sin(2\theta) \right) f'(\theta)
+ 1 = 0 \ .
$$
Now recall that the general term of the solution to the ODE:
$$ a y' + b y = c \ ,$$
is given by:
$$ y(x)
   =
   \exp\left( -\int_{x_0}^x \frac{b}{a} \right)
   \left[ Cste + \int_{x_0}^x du \ \frac{c(u)}{a(u)}
   \exp\left( \int_{x_0}^u \frac{b}{a} \right) \right]
   \ .
$$
Here:
$$ \exp\left( -\int_{\frac{\pi}{2}}^x \frac{b}{a} \right)
   =
   \sin(x)
   e^{ \frac{2w}{\sqrt\gamma} \cotan(x) }
   \ ,
$$
and hence:
$$
   f'(x)
   =
   \sin(x)
   e^{ \frac{2w}{\sqrt\gamma} \cotan(x) }
   \left( 
   Cste
   -
   \frac{2}{\gamma}
   \int_{\frac{\pi}{2}}^x \frac{du}{\sin^3 u} e^{-\frac{2w}{\sqrt\gamma} \cotan(u) }
   \right)
   \ .
$$

The only solution which remains bounded and integrable as $x \rightarrow 0^+$ is:
\begin{align*}
   f'(x) = &
   -
   \frac{2}{\gamma}
   \sin(x)
   e^{ \frac{2w}{\sqrt\gamma} \cotan(x) }
   \int_{0}^x \frac{du}{\sin^3 u} e^{-\frac{2w}{\sqrt\gamma} \cotan(u) }
   \ .
\end{align*}

In the end, the quantity of interest is:
$$
   \E_0\left( \tau_{[0, \varepsilon]} \right)
   =
   \frac{2}{\gamma}
   \int_0^\varepsilon du \ 
   \sin(u)
   e^{ \frac{2w}{\sqrt\gamma} \cotan(u) }
   \int_{0}^u \frac{dv}{\sin^3 v} e^{-\frac{2w}{\sqrt\gamma} \cotan(v) }
   \ .
$$

The asymptotic expansion as $\gamma \rightarrow \infty$ is readily seen after performing integrations by parts:
\begin{align*}
   \E_0\left( \tau_{[0, \varepsilon]} \right)
= &
   \frac{2}{\gamma} \frac{\sqrt{\gamma}}{2 w}
   \int_0^\varepsilon du 
   \sin(u)
   e^{ \frac{2w}{\sqrt\gamma} \cotan(u) }
   \int_{0}^u \frac{dv}{\sin v} \partial_v e^{-\frac{2w}{\sqrt\gamma} \cotan(v) }
   \\
= &
   \frac{1}{\sqrt{\gamma}w}
   \int_0^\varepsilon du 
   \sin(u)
   e^{ \frac{2w}{\sqrt\gamma} \cotan(u) }\left[\frac{1}{\sin (u)}e^{ -\frac{2w}{\sqrt\gamma} \cotan(u) }\right]   \\
  &
  -
   \frac{1}{\sqrt{\gamma}w} 
   \int_0^\varepsilon du 
   \sin(u)
   e^{ \frac{2w}{\sqrt\gamma} \cotan(u) }\int_{0}^u dv\frac{-\cos(v)}{\sin^2 v} e^{-\frac{2w}{\sqrt\gamma} \cotan(v) }  \\
= &
   \frac{\varepsilon}{\sqrt{\gamma}w}
   +
   \frac{1}{2 w^2} \int_0^\varepsilon du 
   \sin(u)
   e^{ \frac{2w}{\sqrt\gamma} \cotan(u) }
   \int_0^u dv \cos(v) \partial_v e^{ -\frac{2w}{\sqrt\gamma} \cotan(u) }\\
= &
   \frac{\varepsilon}{\sqrt{\gamma}w}
   +
   \frac{1}{2 w^2} \int_0^\varepsilon du 
   \sin(u)
   e^{ \frac{2w}{\sqrt\gamma} \cotan(u) }\left[\cos(u)e^{ -\frac{2w}{\sqrt\gamma} \cotan(u) }\right]\\
  &
  +\frac{1}{2 w^2} \int_0^\varepsilon du \sin(u)
   e^{ \frac{2w}{\sqrt\gamma} \cotan(u) }\int_{0}^u dv\sin(v) e^{-\frac{2w}{\sqrt\gamma} \cotan(v) }\\
= &
   \frac{\varepsilon}{\sqrt{\gamma}w}
   +
   \frac{1}{8 w^2} \left( 1 - \cos(2 \varepsilon) \right)\\
  &
  +\frac{1}{2 w^2} \int_0^\varepsilon du \sin(u)
   e^{ \frac{2w}{\sqrt\gamma} \cotan(u) }\int_{0}^u dv\sin(v) e^{-\frac{2w}{\sqrt\gamma} \cotan(v) }\\
\leq &
   \frac{\varepsilon}{\sqrt{\gamma}w}
   +
   \frac{1}{8 w^2} \left( 1 - \cos(2 \varepsilon) \right)
   +
   \frac{1}{2 w^2}\int_0^\varepsilon du \ u \sin^2(u) \ .
\end{align*}
The result follows.
\end{proof}

\bigskip

{\bf Acknowledgements:}
R.C. and C.P. thank Michel Bauer and Denis Bernard for fruitful conversations.

Also during the finalization of our paper, we learned that Nicolas Fournier and Camille Tardif have a similar approach for the case of Rabi oscillations.

C.B. and R.C. have been supported by the grant ANR-15-CE40-0020 LSD of the French National Research Agency (ANR).

R.C. has been supported by the grant ANR-18-CE40-0006 MESA of the French National Research Agency (ANR).

J. N. has been supported by the start up grant U100560-109 from the University of Bristol and a Focused Research Grant from the Heibronn Institute for Mathematical Research.

C.P.\ has been supported by the ANR project StoQ ANR-14-CE25-0003 and Project Markov of Labex CIMI. 

\bibliographystyle{halpha}
\bibliography{QuantumCollapse-1d.bib}

\newcommand{\etalchar}[1]{$^{#1}$}
\begin{thebibliography}{BCF{\etalchar{+}}19}

\bibitem[BB14]{BD14-JP}
Michel Bauer and Denis Bernard.
\newblock Real time imaging of quantum and thermal fluctuations: the case of a
  two-level system.
\newblock {\em Lett. Math. Phys.}, 104(6):707--729, 2014.

\bibitem[BB18]{BB18}
Michel Bauer and Denis Bernard.
\newblock Stochastic spikes and strong noise limits of stochastic differential
  equations.
\newblock In {\em Annales Henri Poincar{\'e}}, volume~19, pages 653--693.
  Springer, 2018.

\bibitem[BBT15]{BBT15}
Michel Bauer, Denis Bernard, and Antoine Tilloy.
\newblock Computing the rates of measurement-induced quantum jumps.
\newblock {\em Journal of Physics A: Mathematical and Theoretical},
  48(25):25FT02, 2015.

\bibitem[BBT16]{BBT16}
Michel Bauer, Denis Bernard, and Antoine Tilloy.
\newblock Zooming in on quantum trajectories.
\newblock {\em Journal of Physics A: Mathematical and Theoretical},
  49(10):10LT01, 2016.

\bibitem[BCF{\etalchar{+}}19]{BCFFS17}
Miguel Ballesteros, Nick Crawford, Martin Fraas, J{\"u}rg Fr{\"o}hlich, and
  Baptiste Schubnel.
\newblock Perturbation theory for weak measurements in quantum mechanics,
  systems with finite-dimensional state space.
\newblock In {\em Annales Henri Poincar{\'e}}, volume~20, pages 299--335.
  Springer, 2019.

\bibitem[Bil13]{B13}
Patrick Billingsley.
\newblock {\em Convergence of probability measures}.
\newblock John Wiley \& Sons, 2013.

\bibitem[FW12]{FW12}
Mark~I. Freidlin and Alexander~D. Wentzell.
\newblock {\em Random Perturbations of Dynamical Systems}.
\newblock Springer Berlin Heidelberg, 2012.

\bibitem[Hal13]{H13}
Brian~C. Hall.
\newblock {\em Quantum theory for mathematicians}, volume 267.
\newblock Springer, 2013.

\bibitem[HR06]{HR06}
Serge Haroche and Jean-Michel Raimond.
\newblock {\em Exploring the quantum}.
\newblock Oxford Graduate Texts. Oxford University Press, Oxford, 2006.
\newblock Atoms, cavities and photons.

\bibitem[KL19]{KL18}
Martin Kolb and Matthias Liesenfeld.
\newblock Stochastic spikes and poisson approximation of one-dimensional
  stochastic differential equations with applications to continuously measured
  quantum systems.
\newblock In {\em Annales Henri Poincar{\'e}}, volume~20, pages 1753--1783.
  Springer, 2019.

\bibitem[Kle05]{K05}
Fima~C. Klebaner.
\newblock {\em Introduction to stochastic calculus with applications}.
\newblock World Scientific Publishing Company, 2005.

\bibitem[LB11]{L11}
Michel Le~Bellac.
\newblock {\em Quantum physics}.
\newblock Cambridge University Press, 2011.

\bibitem[MS77]{MS77}
Baidyanath Misra and E.C.~George Sudarshan.
\newblock {The Zeno's paradox in quantum theory}.
\newblock {\em Journal of Mathematical Physics}, 18(4):756--763, 1977.

\bibitem[Mun00]{Munkres2000}
James~R. Munkres.
\newblock {\em Topology}.
\newblock Prentice Hall, 2000.

\bibitem[RY99]{RY}
Daniel Revuz and Marc Yor.
\newblock {\em Continuous martingales and {B}rownian motion}, volume 293 of
  {\em Grundlehren der Mathematischen Wissenschaften [Fundamental Principles of
  Mathematical Sciences]}.
\newblock Springer-Verlag, Berlin, third edition, 1999.

\bibitem[TBB15]{TBB15}
Antoine Tilloy, Michel Bauer, and Denis Bernard.
\newblock Spikes in quantum trajectories.
\newblock {\em Physical Review A}, 92(5):052111, 2015.

\bibitem[Whi02]{W02}
Ward Whitt.
\newblock {\em Stochastic-process limits: an introduction to stochastic-process
  limits and their application to queues}.
\newblock Springer Science \& Business Media, 2002.

\bibitem[WM10]{WM10}
Howard~M. Wiseman and Gerard~J. Milburn.
\newblock {\em Quantum measurement and control}.
\newblock Cambridge University Press, Cambridge, 2010.

\end{thebibliography}

\end{document}